\documentclass{amsart}

\usepackage{mathrsfs}
\usepackage{amsfonts}
\usepackage{amssymb,amsmath}
\usepackage{amsthm}

\usepackage{graphics}
\usepackage{graphicx}
\usepackage{epsfig}
\usepackage[bookmarks=true]{hyperref}
\usepackage{multirow} 
\usepackage{xcolor}
\usepackage{makecell}

\newtheorem{theorem}{Theorem}[section]

\theoremstyle{definition}

\theoremstyle{remark}
\newtheorem{remark}[theorem]{Remark}
\theoremstyle{assumption}

\numberwithin{equation}{section}

\theoremstyle{problem}

\usepackage{epsfig}
\usepackage{bbm}
\hypersetup{hidelinks}
\usepackage{algorithm} 
\usepackage{algorithmic} 

\ifpdf
\DeclareGraphicsExtensions{.eps,.pdf,.png,.jpg}
\else
\DeclareGraphicsExtensions{.eps}
\fi

\begin{document}

\title[Adaptive augmented regularization method]{An adaptive augmented regularization method and its applications}

\author[J. Jia]{Junxiong Jia}
\address{School of Mathematics and Statistics,
Xi'an Jiaotong University,
 Xi'an
710049, China; }
\email{jjx323@mail.xjtu.edu.cn}

\author[Q. Sun]{Qihang Sun}
\address{School of Mathematics and Statistics,
Xi'an Jiaotong University,
 Xi'an
710049, China;}
\email{tjlxsunqihang@126.com}

\author[B. Wu]{Bangyu Wu}
\address{School of Mathematics and Statistics,
Xi'an Jiaotong University,
 Xi'an
710049, China;}
\email{bangyuwu@xjtu.edu.cn}

\author[J. Peng]{Jigen Peng}
\address{School of Mathematics and Statistics,
Xi'an Jiaotong University,
 Xi'an
710049, China;}
\email{jgpen@mail.xjtu.edu.cn}






\begin{abstract}
Regularization method and Bayesian inverse method are two dominating ways for solving inverse problems generated from various
fields, e.g., seismic exploration and medical imaging.
The two methods are related with each other by the MAP estimates of posterior probability distributions.
Considering this connection, we construct a prior probability distribution with several hyper-parameters and provide the relevant Bayes' formula, then
we propose a corresponding adaptive augmented regularization model (AARM).
According to the measured data, the proposed AARM can adjust its form to various regularization models at each discrete point of the estimated function,
which makes the characterization of local smooth properties of the estimated function possible.
By proposing a modified Bregman iterative algorithm, we construct an alternate iterative algorithm to solve the AARM efficiently.
In the end, we provide some numerical examples which clearly indicate that the proposed AARM can generates a favorable result for some examples
compared with several Tikhonov and Total-Variation regularization models.
\end{abstract}

\maketitle


\section{Introduction}

Consider the following abstract formulation for noisy indirect observations of a function $f$,
\begin{align}\label{1AbstractForm}
d = \mathcal{F}(f) + \epsilon,
\end{align}
where $f$ is a function in some Banach space $X$, $d \in \mathbb{R}^{m}$ represents the measurement data, $\epsilon \in \mathbb{R}^{m}$ stands for the measurement noise
and $\mathcal{F}\, : \, X \rightarrow \mathbb{R}^{m}$ represents some forward map, e.g., convolution operator, acoustic wave equation and diffusion equation.
The inverse problem is to estimate $f$ from the noisy data $d$ which include many types of problems such as deblurring \cite{Kindermann2005Deblurring},
inverse source problem \cite{Bao2013RandomSource} and full waveform inversion \cite{M2015Full}.

There are two main methods for solving inverse problems: one is the regularization method, another one is the Bayesian inverse method.
The two methods are closely linked with each other by the maximum a posteriori (MAP) estimate of the posterior probability distribution.
Especially, some types of Tikhonov regularization model can be seen as the MAP estimate in the Bayesian inverse framework
with Gaussian prior and Gaussian noise assumptions \cite{MAP_detail,MAP_Besov,Tarantola2005book}.
In this paper, we will propose a novel regularization model which is enlightened by the Bayesian inverse method.
In order to state the motivations clearly, let us recall some important aspects of the regularization method in the following.

Tikhonov regularization is one of the most popular methods for solving inverse problems, which formulate inverse problems
as minimization problems with residual term and regularization term \cite{Engl2000Regularization}.
For the reader's convenience, we list two specific models of the general Tikhonov regularization model as follows
\begin{align}\label{1TikhonovModel1}
\min_{f} \left\{ \|d - \mathcal{F}(f)\|_{2}^{2} + \lambda \|\nabla^{2} f\|_{2}^{2} \right\},
\end{align}
\begin{align}\label{1TikhonovModel2}
\min_{f} \left\{ \|d - \mathcal{F}(f)\|_{2}^{2} + \lambda \|f\|_{2}^{2} \right\},
\end{align}
where $\|\cdot\|_{2}$ denotes $L^{2}$ norm for functions and represents $\ell^{2}$ norm for vectors and $\lambda$ is a given constant.
Model (\ref{1TikhonovModel1}) and (\ref{1TikhonovModel2}) will be used in the following statements.
There are already numerous algorithms for solving
Tikhonov regularization models, however, it always over smoothing discontinuous parts of the estimated function $f$ \cite{Osher2005An}.
For a recent progress, Calvetti et al. \cite{Calvetti2014Variable} propose a new type of Tikhonov regularization model
based on Bayesian inverse framework, which can capture the highly oscillation parts of a function.

In order to overcome the drawbacks of Tikhonov regularization method, Total-Variation (TV) regularization has been proposed
by Rudin et al. in \cite{RUDIN1992259} for the problems of image denoising.
We also provide the TV regularization model used in this paper as follows
\begin{align}\label{1TVModel}
\min_{f} \left\{ \|d - \mathcal{F}(f)\|_{2}^{2} + \lambda \|f\|_{TV} \right\},
\end{align}
where $\|\cdot\|_{TV}$ represents Total-Variation norm and $\lambda$ is a given constant.
This model can capture the discontinuous parts of a function, however, it will lead to staircasing effect which
means that this model tends to find a piecewise-constant function \cite{LI2010870}.
When the original function is a smooth and slowly changed function, staircasing effect will make the recovered function unacceptable.

A natural question is how to construct a new regularization model that has fine performance on different parts of a function.
Specifically speaking, for a function shown in Figure \ref{ComplexSignal},
we need the regularization model generates a similar result as the TV regularization model for the blue part (solid line).
For the green part (dashed line), we would like the new regularization model performs similar to the Tikhonov regularization model (\ref{1TikhonovModel1}).
At last, we expect that the new model generates a similar result as the Tikhonov regularization model (\ref{1TikhonovModel2}) for the red part (dash-dotted line).
\begin{figure}[htb]
  \centering
  \includegraphics[width = 0.65\textwidth]{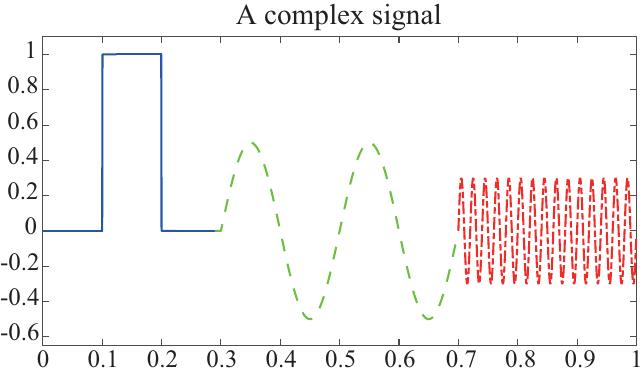}\\
  \caption{A signal with discontinuous part (blue, solid line), continuous and slowly changing part (green, dashed line),
  continuous and fast changing part (red, dash-dotted line).
  }\label{ComplexSignal}
\end{figure}
In order to attain this goal, we need the new model to alter its behavior between different regularization models on each particular region.
There are already many investigations in this direction. The basic variable exponent type regularization model has the following form
\begin{align}\label{1PTVModel}
\min_{f} \left\{ \|d - \mathcal{F}(f)\|_{2}^{2} + \lambda \int_{\Omega}|\nabla f(x)|^{p(|\nabla f|)}dx \right\},
\end{align}
where $p(\cdot)$ is a scalar function which tends to $1$ if $|\nabla f|$ goes to $\infty$, and tends to $2$ if $|\nabla f|$ goes to $0$.
For investigations on this model and its various variations, we refer to \cite{Blomgren1999Variational,HARJULEHTO201356,HARJULEHTO2008174,LI2010870}.
However, there are two main problems for the regularization model (\ref{1PTVModel}):
\begin{enumerate}
  \item How to specify the parameter $\lambda$ which balance the effects of the residual term and the regularization term;
  \item How to design a scalar function $p(\cdot)$ to adjust the regularization term efficiently.
\end{enumerate}
For the parameter $\lambda$, it can be determined by using Morozov's discrepancy principle \cite{DisBayesian1}.
However, no matter which value of $\lambda$ to be chosen, $\lambda$ balance the residual term
and the regularization term in the whole domain of $f$.
Hence, it can not balance the two terms in some local regions. The blue part, the green part and the red part in Figure \ref{ComplexSignal}
obviously need different values of $\lambda$ to obtain optimal estimations.
For the scalar function $p(\cdot)$, it depends on an unknown value $|\nabla f|$ which usually can only be estimated roughly.
The dependence of $p(\cdot)$ on $|\nabla f|$ also leads difficulties for designing efficient iterative algorithms.
From the perspective of Bayesian inverse method, a variable-order Besov prior probability measure has been constructed in \cite{Junxiong2016variable}
to achieve a similar aim as model (\ref{1PTVModel}).
However, no practical algorithms have been proposed, which is also one of the motivations for our work.

In this paper, enlightened by the Bayesian inverse method, we propose an adaptive augmented regularization model (AARM) to
overcome the aforementioned two difficulties for one-dimensional functions.
The Bayesian inverse method has also been employed by Jin and Zou \cite{Jin2009IP,Jin2010JCP} for developing an augmented Tikhonov regularization method which can determine
regularization parameters from data.
Now, let us provide a short explanation of our main idea.
Firstly, we will construct our model through Bayesian inverse framework,
then, we propose the corresponding regularization model by the MAP estimate of the posterior probability distribution.
In this procedure, the key point is to construct an appropriate prior probability distribution which can
generate functions similar to the function shown in Figure \ref{ComplexSignal}.
Autoregressive Markov models are employed to construct the prior probability distribution.
Different to the classical autoregressive Markov models, a vector of hyper-parameters $\theta$ are employed
to integrate two autoregressive Markov models with different smooth levels.
So, the parameter $\theta$ adjust the smooth levels of regularization terms, and the value of each component of $\theta$
reflects the smoothness of the estimated function $f$ in each discrete point.
Relying on $\theta$, a mechanism has been designed to alternate the model
between TV and Tikhonov regularization models in different regions of the estimated function $f$, which solves Problem (2) mentioned in the previous paragraph.
Because we consider the problem under Bayesian inverse framework, only probability distributions of hyper-parameters $\theta$ and $\gamma$
are specified. Therefore, the hyper-parameters have the ability of adjusting its values in each discrete point according to the measured data,
which solve Problem (1) mentioned in the previous paragraph.

The organization of this paper is as follows. In Section 2, we construct a spatially adaptive prior probability distribution based
on autoregressive Markov models with different smooth levels. In the construction, several hyper-parameters have been introduced
and a mechanism has been designed to alternate between TV and Tikhonov regularization models.
In Section 3, through the MAP estimate of posterior probability distribution, an adaptive augmented regularization model (AARM) has been proposed.
Then, we design an alternate iterative algorithm to solve the proposed AARM.
In each alternate iterative process, the first two sub-optimization problems have been solved based on a modified Bregman iterative algorithm and
the third sub-optimization problem can be solved simply by sufficient conditions of optimization points.
At the end of this section, some preliminary theoretical investigations are provided.
In Section 4, we show some numerical results of deconvolution problems obtained by using different methods,
which illustrate the effectiveness of the proposed AARM.
In Section 5, we summarize the main points and provide some further problems.

\section{Inverse problems in Bayesian statistical framework}\label{2sectionTitle}

In this section, we assume $d \in \mathbb{R}^{m}$. Given two real numbers $a$ and $b$, let $S_{n}$ stands for a sample operator which is defined as follows
\begin{align}\label{2SampleOpe}
(S_{n}f) = (f(x_{0}), f(x_{1}), \ldots, f(x_{n})), \quad \text{where} \,\,  x_{j} = a + j\Delta x, \,\, j = 0,1,\ldots,n,
\end{align}
where $x_{j} \in [a,b]$ and $\Delta x = \frac{b-a}{n}$.
Denote $f_{j} := f(x_{j})$ with $j = 0,1,\ldots,n$, $f_{d} := (S_{n}f) \in \mathbb{R}^{n+1}$, then (\ref{1AbstractForm}) can be reformulated as follows
\begin{align}\label{2AbstractForm}
d = \mathcal{F}_{mn}(f_{d}) + \epsilon,
\end{align}
where $\mathcal{F}_{mn}$ stands for the discretized version of the forward operator and $\epsilon \in \mathbb{R}^{m}$ represents some random noise.

In the following, we denote $\|\cdot\|_{2}$ as $L^{2}$ norm for functions and $\ell^{2}$ norm for vectors.
Similarly, $\|\cdot\|_{1}$ denotes $L^{1}$ norm for functions and $\ell^{1}$ norm for vectors.

\subsection{Spatially adaptive prior model}

Denote $\mathcal{D}_{s}$ to be a one-dimensional standard Gaussian distribution $\textbf{Gaussian}(0,1)$ when $s=1$
and a Laplace distribution $\textbf{Laplace}(0,2)$ with location paremeter $0$ and diversity parameter $2$ when $s=2$.
For constructions of the prior probability distribution, a frequently used method consists of autoregressive Markov (AR) models of the form
\begin{align}\label{21Autoregressive}
f_{j}  = \sum_{k = 1}^{p} \alpha_{k} f_{j-k} + \sqrt{\gamma_{j}}W_{j}, \quad W_{j} \sim \mathcal{D}_{s}, \quad 0\leq j\leq n,
\end{align}
where the order $p$ and the coefficient $\alpha_{k}$ are given,
and $\gamma_{j}$ is the variance of the $j$th random variable $W_{j}$.
In order to have an explicit definition, we should specify the values of $f_{j}$ for $j < 0$ in formula (\ref{21Autoregressive})
which can be given according to the requirements of a concrete problem.
Here, for simplicity, we assume
\begin{align*}
f_{j} = 0 \text{ almost certainly for } j < 0.
\end{align*}
In this case, the main idea can be shown clearly and it may be easily extended to a generalized version according to the requirements of some specific problems.

In model (\ref{21Autoregressive}), the parameter $p$ reflects the regularity properties of functions in an intuitive sense.
For two parameters $p,q$ with $p > q$, we consider two AR models,
\begin{align}\label{21twoModels1}
f_{j} & = \sum_{k = 1}^{p} \alpha_{k}f_{j-k} + \sqrt{\gamma_{j}^{(1)}}W_{j}^{(1)}, \quad W_{j}^{(1)} \sim \mathcal{D}_{s},
\end{align}
\begin{align}\label{21twoModels2}
f_{j} & = \sum_{k = 1}^{q} \beta_{k}f_{j-k} + \sqrt{\gamma_{j}^{(2)}}W_{j}^{(2)}, \quad W_{j}^{(2)} \sim \mathcal{D}_{s},
\end{align}
$j = 0,1,\ldots,n$.
Through some simple calculations, we could introduce two matrixes $L_{p}$ and $L_{q}$ to reformulate (\ref{21twoModels1}) and (\ref{21twoModels2}) as follows
\begin{align}\label{21matrixTwo}
L_{p}f_{d} = D_{\gamma^{(1)}}^{1/2}W^{(1)}, \quad L_{q}f_{d} = D_{\gamma^{(2)}}^{1/2}W^{(2)},
\end{align}
where
\begin{align*}
D_{\gamma^{(j)}} =
\left(
  \begin{array}{ccc}
    \gamma_{0}^{(j)} & \cdots & 0 \\
    \vdots & \ddots & \vdots \\
    0 & \cdots & \gamma_{n}^{(j)} \\
  \end{array}
\right), \quad W^{(j)} = (W_{0}^{(j)},\ldots,W_{n}^{(j)})^{T},\quad \text{with  } j = 1,2.
\end{align*}
Here and in the following, for a diagonal matrix $A = \text{diag}(a_{0}, a_{1}, \ldots, a_{n})$, the notation $A^{\alpha}$ means
$\text{diag}(a_{0}^{\alpha}, a_{1}^{\alpha}, \ldots, a_{n}^{\alpha})$ for $\alpha \in \mathbb{R}$.

With these preparations, a weighted variable order autoregressive model could be proposed thorough similar ideas shown in \cite{Calvetti2014Variable}.
Now, introducing a sequence of parameters $\theta_{j} \in [0,1]$ with $j = 0,1,\ldots,n$, we can construct
\begin{align}\label{21combineModel}
f_{j} = \sum_{k = 1}^{p}\left( (1-\theta_{j})\alpha_{k} + \theta_{j}\beta_{k} \right)f_{j-k} + \sqrt{\gamma_{j}}W_{j}, \quad W_{j} \sim \mathcal{D}_{s},
\end{align}
where
\begin{align*}
\gamma_{j} = (1-\theta_{j})^{2}\gamma_{j}^{(1)} + \theta_{j}^{2}\gamma_{j}^{(2)},
\end{align*}
and $\beta_{k} = 0$ for $k > q$.
Define
\begin{align}\label{21combineMatrix}
L_{\theta} := (1-D_{\theta})L_{p} + D_{\theta}L_{q}
\end{align}
with
\begin{align*}
D_{\theta} =
\left(
  \begin{array}{ccc}
    \theta_{0} & \cdots & 0 \\
    \vdots & \ddots & \vdots \\
    0 & \cdots & \theta_{n} \\
  \end{array}
\right).
\end{align*}
Now, the weighted variable order autoregressive model (\ref{21combineModel}) can be written as
\begin{align}\label{21comSimpleModel}
L_{\theta}f_{d} = D_{\gamma}^{1/2}W
\end{align}
with $W = (W_{0}, W_{1}, \ldots, W_{n})^{T}$.


Compared with the Gaussian distribution, the Laplace distribution is a heavy-tailed distribution which may more suitable
for characterizing discontinuous properties of functions.
Based on different properties of the Gaussian and the Laplace distributions, we need to model $W_{j}$ as a standard Gaussian random variable
if $f_{j}$ depends smoothly on $f_{j-1},\ldots,f_{j-p}$ that is to say the function $f$ or the derivative of the function $f$ or other
calculations of the function $f$ vary smoothly.
Otherwise, we need to model $W_{j}$ as a Laplace random variable to represent discontinuous changes of some calculations
(e.g., first-order derivative, second-order derivative) of the function $f$.
We find that the hyper-parameter $\theta$ in formula (\ref{21comSimpleModel}) could provide us an efficient way to distinguish
continuous and discontinuous parts of the function $f$.
Specifically speaking, we introduce a parameter $p^{\theta} = (p_{0}^{\theta}, p_{1}^{\theta}, \ldots, p_{n}^{\theta})$ defined as follows
\begin{align}\label{21defineP}
p_{j}^{\theta} =
\left\{ \begin{aligned}
& \, 1,\quad |\theta_{j} - \theta_{j-1}| > T_{s}, \\
& \, 2, \quad |\theta_{j} - \theta_{j-1}| \leq T_{s},
\end{aligned}
\right.
\end{align}
where $\theta_{-1} := 0$ and $T_{s}$ is a threshold value. For simplicity, we provide an intuitive way to determine $T_{s}$ as follows
\begin{align}\label{21determineTs}
T_{s} := \text{max}\left( \text{min}(M\cdot\text{mean}(|\theta_{d}|), \, r\cdot\text{max}(|\theta_{d}|)) , \, m\cdot\text{mean}(|\theta_{d}|) \right),
\end{align}
where $M$, $m$ and $r$ are three real numbers with $0 < m < M < \infty$ and $1/2 < r < 1$ and
\begin{align*}
& |\theta_{d}| := (|\theta_{1} - \theta_{0}|, |\theta_{2} - \theta_{1}|, \ldots, |\theta_{n} - \theta_{n-1}|), \\
& \quad\quad\,\,\,  \text{mean}(|\theta_{d}|) := \frac{1}{n}\sum_{k = 1}^{n}|\theta_{j} - \theta_{j-1}|.
\end{align*}
With (\ref{21defineP}) and (\ref{21determineTs}), we can determine $W_{j}\, (j = 0,1,\ldots,n)$ as follows
\begin{align}\label{21disWj}
W_{j} \sim
\left\{ \begin{aligned}
& \, \textbf{Gaussian}(0,1), \quad p_{j}^{\theta} = 2, \\
& \, \textbf{Laplace}(0,2), \,\,\,\, \quad p_{j}^{\theta} = 1.
\end{aligned}\right.
\end{align}
\begin{remark}\label{21explain1}
$|\theta_{d}|$ is a vector reflects the changes between two autoregressive models with parameters $p$ and $q$ ($p > q$) respectively.
The two autoregressive Markov models represent our different beliefs on the smoothness of the function $f$, so the value of each component of the vector $|\theta_{d}|$
reflects the changes of smoothness of the function $f$.
Large changes of the smoothness of the function $f$ imply that we need to use the Laplace random variables.
On the contrary, small changes indicate that the Gaussian random variables are an appropriate choice.
Formula (\ref{21disWj}) just reflects these considerations.
\end{remark}

At this stage, we can easily find that
\begin{align}\label{21modelAdd}
D_{\gamma}^{-1/2}L_{\theta}f_{d} = W,
\end{align}
and
\begin{align}\label{21modelDistribution}
\pi_{\text{prior}}(f_{d} \, | \, \theta, \gamma) \propto \text{Det}\left( D_{\gamma}^{-1/2}L_{\theta} \right)
\exp\left( - \frac{1}{2}\|D_{\gamma}^{-1/2}L_{\theta}f_{d}\|_{p^{\theta}}^{p^{\theta}} \right),
\end{align}
where
\begin{align}\label{21definePtheta}
\|D_{\gamma}^{-1/2}L_{\theta}f_{d}\|_{p^{\theta}}^{p^{\theta}} := \sum_{k = 0}^{n} |D_{\gamma}^{-1/2}L_{\theta}f_{d}|^{p_{j}^{\theta}}.
\end{align}
Since
\begin{align*}
\text{Det}(D_{\gamma}^{-1/2}L_{\theta}) = \text{Det}(D_{\gamma}^{-1/2}) \text{Det}(L_{\theta}) = \prod_{j = 0}^{n} \gamma_{j}^{-1/2},
\end{align*}
we can reduce formula (\ref{21modelDistribution}) as follows
\begin{align}\label{21modelDistriFinal}
\pi_{\text{prior}}(f_{d} \, | \, \theta, \gamma) \propto
\exp\left( - \frac{1}{2}\|D_{\gamma}^{-1/2}L_{\theta}f_{d}\|_{p^{\theta}}^{p^{\theta}} - \frac{1}{2}\sum_{k = 0}^{n}\text{log}\gamma_{j} \right).
\end{align}

In order to form a useful prior model, we should specify the statistical properties of the hyper-parameters $\gamma$ and $\theta$ appearing in
formula (\ref{21modelDistriFinal}).
Each component of the hyper-parameter $\gamma = (\gamma_{0}, \ldots, \gamma_{n})$ stands for the variance of each component of the random vector
$W_{\gamma} := (\sqrt{\gamma_{0}}W_{0}, \ldots, \sqrt{\gamma_{n}}W_{n})$.
Our beliefs on the random vector $W_{\gamma}$ determine the statistical properties of $\gamma$.

\textbf{Belief 1 of $\gamma$}: Our model (\ref{21modelDistriFinal}) describe the function $f$ quite well except at a few places where it may have
relatively large jumps.

In order to reflect this belief, we can assume that each component of $f_{d}$ are independent, and use the family of generalized gamma distributions
$\textbf{GenGamma}(r,\beta,\bar{\gamma})$ \cite{Calvetti2014Variable,Shin2005Statistical}, defined as
\begin{align}\label{21beliefgamma1}
\begin{split}
\pi_{\text{hyper},1}(\gamma) \propto \exp\left( -\sum_{j = 0}^{n}\left( \frac{\gamma_{j}}{\bar{\gamma}} \right)^{r} + (r\beta - 1)\sum_{j=0}^{n}\text{log}\gamma_{j} \right),
\end{split}
\end{align}
where $r, \beta, \bar{\gamma}$ are specified appropriately.

\textbf{Belief 2 of $\gamma$}: Our model (\ref{21modelDistriFinal}) describe the function $f$ quite well at all points.

With this belief, we also assume that each component of $f_{d}$ are independent, and employ the family of Gaussian distributions
$\textbf{Gaussian}(0,\eta)$ to define
\begin{align}\label{21beliefgamma2}
\begin{split}
\pi_{\text{hyper},1}(\gamma) \propto \exp\left( - \frac{1}{2\eta}\sum_{j = 0}^{n}\gamma_{j}^{2} \right)\mathbbm{1}_{[0,\infty)^{n+1}}(\gamma)
\end{split}
\end{align}
with appropriately specified $\eta$.

Similarly, statistical properties of the parameter $\theta$ depend on our beliefs of the changes of functions.

\textbf{Belief 1 of $\theta$}: The variations between the two autoregressive models are slow.

Define the discrete finite difference operator as follows
\begin{align}\label{21defineL1}
L_{1} =
\left(
  \begin{array}{cccc}
    1 & 0 & \cdots & 0 \\
    -1 & 1 & \cdots & 0 \\
    \vdots & \vdots & \ddots & \vdots \\
    0 & \cdots & -1 & 1 \\
  \end{array}
\right),
\end{align}
then we can use the following model
\begin{align}\label{21beliefTheta1}
\pi_{\text{hyper},2}(\theta) \propto \exp\left( -\frac{1}{2\eta}\|L_{1}\theta\|_{2}^{2} \right)\mathbbm{1}_{[0,1]^{n+1}}(\theta),
\end{align}
where $\mathbbm{1}_{[0,1]^{n+1}}(\cdot)$ is the characteristic function of the hypercube in $\mathbb{R}^{n+1}$
and $\eta$ is given.

\textbf{Belief 2 of $\theta$}: The variations between the two autoregressive models are slow, however, at some points the changes may be large.

With this belief, we can instead the Gaussian distribution by Laplace distribution to use the following model
\begin{align}\label{21beliefTheta2}
\pi_{\text{hyper},2}(\theta) \propto \exp\left( -\frac{1}{2\eta}\|L_{1}\theta\|_{1} \right)\mathbbm{1}_{[0,1]^{n+1}}(\theta),
\end{align}
where $\eta$ is given.

Finally, we propose a \emph{spatially adaptive prior probability distribution} has the following density function
\begin{align}\label{21fullModel}
\pi_{\text{p}}(f_{d}, \gamma, \theta) = \pi_{\text{prior}}(f_{d} \, | \, \theta, \gamma)\pi_{\text{prior},1}(\gamma)\pi_{\text{hyper},2}(\theta).
\end{align}
\begin{remark}\label{21fullModelRemark}
In the prior probability density function (\ref{21fullModel}), we only specify the statistical properties of the hyper-parameters $\gamma$ and $\theta$.
The concrete values of $\gamma$ and $\theta$ will be figured out by incorporating the information of the measured data
which is an advantage for considering inverse problem under Bayesian statistical framework.
In other words, we construct a spatially adaptive prior probability distribution which can
capture the local smoothness of the function $f$ by adjusting the hyper-parameters through some noisy data.
\end{remark}

\subsection{Likelihood and Bayes' formula}

For the additive noise model (\ref{2AbstractForm}), we assume that the noise $\epsilon$ is a realization of
a Gaussian random variable $E$,
\begin{align*}
E \sim \textbf{Gaussian}(0,\Sigma).
\end{align*}
For simplicity, we denote $S$ as the Cholesky factor of the inverse of the covariance,
\begin{align*}
\Sigma^{-1} = S^{T}S.
\end{align*}
Hence, the likelihood density is given by
\begin{align*}
\pi_{\ell}(d \, | \, f_{d}) \propto \exp\left( -\frac{1}{2}\|S(d-\mathcal{F}_{mn}(f_{d}))\|_{2}^{2} \right).
\end{align*}
Through Bayes' formula, the posterior density has the following form
\begin{align}\label{22Posterior}
\pi_{\text{post}}(f_{d},\theta,\gamma \, | \, d) \propto \pi_{\ell}(d \, | \, f_{d})\pi_{\text{prior}}(f_{d} \, | \, \theta,\gamma)
\pi_{\text{hyper},1}(\gamma)\pi_{\text{hyper},2}(\theta).
\end{align}
By formula (\ref{22Posterior}), we can employ algorithms like Markov Chain Monte Carlo (MCMC) to explore the posterior probability distribution
to obtain full information of the function $f$.
However, in some applications, the computational load of the forward problem is too high to make enough samples.
Hence, alternatively, we can solve the following minimization problem
\begin{align}\label{22Map}
\min_{f_{d}, \theta, \gamma} \pi_{\ell}(d \, | \, f_{d})\pi_{\text{prior}}(f_{d} \, | \, \theta,\gamma)
\pi_{\text{hyper},1}(\gamma)\pi_{\text{hyper},2}(\theta),
\end{align}
to obtain the maximum a posterior estimator (MAP estimator) which connects the Bayesian inverse method and the regularization method.

\section{Adaptive augmented regularization method}\label{MAPSection}

In Section \ref{2sectionTitle}, we consider the inverse problem under the Bayesian statistical framework
and propose a spatially adaptive prior probability distribution.
Now, considering the forward map in real applications usually computational intensive, we only consider the MAP estimator in this section.
Based on the MAP estimator, we propose the following \emph{adaptive augmented regularization model} (AARM)
\begin{align}\label{3DDARmodel}
\begin{split}
\min_{f_{d}, \gamma, \theta} & \Bigg\{ \frac{1}{2}\|S(d - \mathcal{F}_{mn}(f_{d}))\|_{2}^{2}
+ \frac{1}{2}\|D_{\gamma}^{-1/2}L_{\theta}f_{d}\|_{p^{\theta}}^{p^{\theta}} + \frac{1}{2}\sum_{k = 0}^{n}\text{log}\gamma_{j} \\
& \quad\quad\quad\quad\quad
- \log(c_{1}\pi_{\text{hyper},1}(\gamma)) - \log(c_{2}\pi_{\text{hyper},2}(\theta)) \Bigg\},
\end{split}
\end{align}
where $c_{1}$ and $c_{2}$ are normalization constants of probability distributions of $\gamma$ and $\theta$.
If we choose formula (\ref{21beliefgamma1}) and formula (\ref{21beliefTheta1}) as the prior assumptions of $\gamma$ and $\theta$,
we can specify model (\ref{3DDARmodel}) as follows
\begin{align}\label{3DDARmodel1}
\begin{split}
\min_{f_{d}, \gamma, \theta\in [0,1]^{n+1}} & \Bigg\{\frac{1}{2}\|S(d - \mathcal{F}_{mn}(f_{d}))\|_{2}^{2}
+ \frac{1}{2}\|D_{\gamma}^{-1/2}L_{\theta}f_{d}\|_{p^{\theta}}^{p^{\theta}} + \frac{1}{2}\sum_{k = 0}^{n}\text{log}\gamma_{j} \\
& \quad\quad
+ \sum_{j = 0}^{n}\left( \frac{\gamma_{j}}{\bar{\gamma}} \right)^{r} - (r\beta - 1)\sum_{j=0}^{n}\text{log}\gamma_{j} + \frac{1}{2\eta}\|L_{1}\theta\|_{2}^{2}\Bigg\}.
\end{split}
\end{align}
In the following, we will develop an algorithm to solve the AARM (\ref{3DDARmodel1}) and provide some preliminary theoretical analysis.
For other assumptions on statistical properties of $\gamma$ and $\theta$ shown in the previous section, algorithms can be developed
similarly, so the details for other cases are omited.

By introducing
\begin{align}\label{3zong}
\begin{split}
T(f_{d},\gamma,\theta) := & \Bigg\{\frac{1}{2}\|S(d - \mathcal{F}_{mn}(f_{d}))\|_{2}^{2}
+ \frac{1}{2}\|D_{\gamma}^{-1/2}L_{\theta}f_{d}\|_{p^{\theta}}^{p^{\theta}} + \frac{1}{2}\sum_{k = 0}^{n}\text{log}\gamma_{j} \\
& \quad\quad
+ \sum_{j = 0}^{n}\left( \frac{\gamma_{j}}{\bar{\gamma}} \right)^{r} - (r\beta - 1)\sum_{j=0}^{n}\text{log}\gamma_{j} + \frac{1}{2\eta}\|L_{1}\theta\|_{2}^{2}\Bigg\},
\end{split}
\end{align}
problem (\ref{3DDARmodel1}) could be written compactly as follows
\begin{align}\label{3DDARmodel2}
\begin{split}
\min_{f_{d}, \gamma, \theta \in [0,1]^{n+1}} T(f_{d}, \gamma, \theta).
\end{split}
\end{align}
For convenience, we introduce the following frequently used notations
\begin{align}\label{3fn}
T_{f}(f_{d},\gamma,\theta) := \frac{1}{2}\|S(d - \mathcal{F}_{mn}(f_{d}))\|_{2}^{2}
+ \frac{1}{2}\|D_{\gamma}^{-1/2}L_{\theta}f_{d}\|_{p^{\theta}}^{p^{\theta}},
\end{align}
\begin{align}\label{3theta}
\begin{split}
T_{\theta}(f_{d},\gamma,\theta) := \frac{1}{2}\|D_{\gamma}^{-1/2}L_{\theta}f_{d}\|_{p^{\theta}}^{p^{\theta}}
+ \frac{1}{2\eta}\|L_{1}\theta\|_{2}^{2},
\end{split}
\end{align}
and
\begin{align}\label{3gamma}
\begin{split}
T_{\gamma}(f_{d},\gamma,\theta) := & \frac{1}{2}\|D_{\gamma}^{-1/2}L_{\theta}f_{d}\|_{p^{\theta}}^{p^{\theta}}
+ \sum_{j = 0}^{n}\left( \frac{\gamma_{j}}{\bar{\gamma}} \right)^{r} - \left(r\beta - \frac{3}{2}\right)\sum_{j=0}^{n}\text{log}\gamma_{j}.
\end{split}
\end{align}

\subsection{Minimization algorithm}

In order to solve the minimization problem (\ref{3DDARmodel2}), we can use the idea of alternate iteration which are shown in Algorithm \ref{alg:A}.
To make the presentation clearly, the stopping criterion and every step of minimization will be discussed separately.

\begin{algorithm}[htb]
\caption{Alternate iterative algorithm}
\label{alg:A}
\begin{algorithmic}
\STATE {(1) Set $k = 0$, $f_{d} = f_{d}^{0}, \gamma = \gamma^{0}, \theta = \theta^{0}$.}
\STATE {(2) Update $(f_{d}^{k}, \gamma^{k}, \theta^{k}) \rightarrow (f_{d}^{k+1}, \gamma^{k+1}, \theta^{k+1})$,}
\STATE {$\quad\,\,\,$ (a) $f_{d}^{k+1} := \text{argmin}_{f_{d}}T_{f_{d}}(f_{d}, \gamma^{k}, \theta^{k})$,}
\STATE {$\quad\,\,\,$ (b) $\theta^{k+1} := \text{argmin}_{\theta}T_{\theta}(f_{d}^{k+1}, \gamma^{k}, \theta)$,}
\STATE {$\quad\,\,\,$ (c) $\gamma^{k+1} := \text{argmin}_{\gamma}T_{\gamma}(f_{d}^{k+1}, \gamma, \theta^{k+1})$.}
\STATE {(3) If convergence criterion is met, stop, else, $k \leftarrow k + 1$ and continue from (2).}
\end{algorithmic}
\end{algorithm}

\subsubsection{Minimization problem for $f_{d}$}\label{subsection311}

The following minimization problem
\begin{align}\label{311miniPro}
\min_{f_{d}} T_{f_{d}}(f_{d}, \gamma^{k}, \theta^{k}),
\end{align}
is an optimization problem with mixed $\ell^{1}$, $\ell^{2}$ regularization terms, which
can not be solved directly by some classical algorithms, e.g., Bregman iterative algorithm \cite{Goldstein2009The}.
Here, we propose a modified Bregman iterative algorithm to solve (\ref{311miniPro}) efficiently.
Before we show this algorithm, let us recall that the Bregman distance associated with a convex functional $F(\cdot)$ between points $f_{d1}$ and $f_{d2}$ is defined as
\begin{align}\label{BregDis1}
D_{F}^{p}(f_{d1},f_{d2}) := F(f_{d1}) - F(f_{d2}) - \langle p, f_{d1} - f_{d2} \rangle,
\end{align}
where $p \in \partial F(f_{d2}) = \left\{ w \, : \, F(f_{d1}) - F(f_{d2}) \geq \langle w, f_{d1} - f_{d2} \rangle,\, \forall\, f_{d1} \right\}$
is the sub-gradient of $F(\cdot)$ at the point $f_{d2}$.
In our problem, we take
\begin{align}\label{Fis1}
F(f_{d}) = \frac{1}{2}\|D_{\gamma^{k}}^{-1/2}L_{\theta^{k}}f_{d}\|_{p^{\theta^{k}}}^{p^{\theta^{k}}}.
\end{align}
Then our optimization problem (\ref{311miniPro}) transforms into
\begin{align}\label{youhuatemp1}
f_{d}^{*} := \mathop{\arg\min}_{f_{d}} F(f_{d}^{*}) + \langle p, f_{d} - f_{d}^{*} \rangle + D^{p}_{F}(f_{d},f_{d}^{*})
+ \frac{1}{2}\|S(d - \mathcal{F}_{mn}(f_{d}))\|_{2}^{2}.
\end{align}
The following Bregman iterative regularization
\begin{align}\label{youhuatemp2}
f^{m+1} = \mathop{\arg\min}_{f_{d}} D_{F}^{p^{m}}(f_{d},f^{m}) + \frac{1}{2}\|S(d - \mathcal{F}_{mn}(f_{d}))\|_{2}^{2}, \quad m = 0,1,2,\ldots,
\end{align}
has been employed by Osher et al. \cite{Osher2005An} to solve (\ref{youhuatemp1}) approximately.

Using the definition of Bregman distance (\ref{BregDis1}), problem (\ref{youhuatemp2}) turns into
\begin{align}\label{youhuatemp3}
f^{m+1} = \mathop{\arg\min}_{f_{d}} F(f_{d}) - F(f^{m}) - \langle p^{m}, f_{d} - f^{m} \rangle + \frac{1}{2}\|S(d - \mathcal{F}_{mn}(f_{d}))\|_{2}^{2}.
\end{align}
From some classical results or Theorem \ref{existenceUnique} proved later, we know that $f^{1}$ is well defined.
Using the optimality of $f^{m+1}$ in (\ref{youhuatemp2}), we have
\begin{align*}
0 \in & \partial F(f^{m+1}) - p^{m} + \mathcal{F}_{mn}^{*}S^{T}S(\mathcal{F}_{mn}(f^{m+1}) - d) \\
& \quad = p^{m+1} - p^{m} + \mathcal{F}_{mn}^{*}S^{T}S(\mathcal{F}_{mn}(f^{m+1}) - d).
\end{align*}
Hence, the iteration direction in the next step in fact has the representation
\begin{align}\label{gradiNext}
p^{m+1} = p^{m} - \mathcal{F}_{mn}^{*}S^{T}S(\mathcal{F}_{mn}(f^{m+1}) - d).
\end{align}
Define
\begin{align}\label{gIte}
\tilde{g}^{1} = d - \mathcal{F}_{mn}(f^{1}), \quad \tilde{g}^{m+1} = \tilde{g}^{m} + d - \mathcal{F}_{mn}(f^{m+1}), \quad m = 1,2,\ldots.
\end{align}
Relying on this expression and some simple computations, we find that
\begin{align*}
& D_{F}^{p^{m}}(f_{d},f^{m}) + \frac{1}{2}\|S(d - \mathcal{F}_{mn}(f_{d}))\|_{2}^{2}  \\
& = F(f_{d}) - F(f^{m}) + \langle p^{m}, f^{m} \rangle - \langle p^{m}, f_{d} \rangle + \frac{1}{2}\|S(d - \mathcal{F}_{mn}(f_{d}))\|_{2}^{2} \\
& = F(f_{d}) - F(f^{m}) + \langle p^{m}, f^{m} \rangle - \langle \mathcal{F}_{mn}^{*}S^{T}S\tilde{g}^{m}, f_{d} \rangle
+ \frac{1}{2}\|S(d - \mathcal{F}_{mn}(f_{d}))\|_{2}^{2}   \\
& = F(f_{d}) + \frac{1}{2}\|S(\tilde{g}^{m} + d - \mathcal{F}_{mn}(f_{d}))\|_{2}^{2} - F(f^{m}) + \langle p^{m}, f^{m} \rangle
- \langle S(\tilde{g}^{m}+d), S\tilde{g}^{m} \rangle.
\end{align*}
The above expression tells us that the optimization problem (\ref{youhuatemp2}) has same structure as that of (\ref{311miniPro}).
Hence, it is well posed by Theorem \ref{existenceUnique}, and therefore the
sequence $\{ f^{m}\, : \, m \in \mathbb{N} \}$ is well defined.
Now, we provide a recursive procedure which can solve (\ref{youhuatemp2}) numerically in Algorithm \ref{alg:Rec}.
Concerning the properties of $\{ f^{m}\,:\, m\in\mathbb{N} \}$ appeared in Algorithm \ref{alg:Rec},
we postpone to show them in Theorem \ref{wellSeq2} in Subsection \ref{theoreticalSection}.
\begin{algorithm}[htb]
\caption{Recursive procedure for $f_{d}^{k+1}$}
\label{alg:Rec}
\begin{algorithmic}
\STATE {\textbf{Set: }
\begin{align*}
& f^{0} \leftarrow f_{d}^{k} \\
& f^{1} = \mathop{\arg\min}_{f_{d}} \Bigg\{ \frac{1}{2}\|S(d - \mathcal{F}_{mn}(f_{d}))\|_{2}^{2}
+ \frac{1}{2}\|D_{\gamma^{k}}^{-1/2}L_{\theta^{k}}f_{d}\|_{p^{\theta^{k}}}^{p^{\theta^{k}}} \Bigg\}, \\
& \tilde{g}^{1} = d - \mathcal{F}_{mn}(f^{1}),
\end{align*}
}
\STATE {\textbf{Repeat: }
\begin{align*}
& f^{m+1} = \mathop{\arg\min}_{f_{d}} \Bigg\{ \frac{1}{2}\|S(\tilde{g}^{m} + d - \mathcal{F}_{mn}(f_{d}))\|_{2}^{2}
+ \frac{1}{2}\|D_{\gamma^{k}}^{-1/2}L_{\theta^{k}}f_{d}\|_{p^{\theta^{k}}}^{p^{\theta^{k}}} \Bigg\}, \\
& \tilde{g}^{m+1} = \tilde{g}^{m} + d - \mathcal{F}_{mn}(f^{m+1}),
\end{align*}
}
\STATE {\textbf{Until: } Some stopping conditions is satisfied.}
\end{algorithmic}
\end{algorithm}

For simplicity, we denote $g^{m} = \tilde{g}^{m} + d$ for $m = 0,1,\ldots.$
Then, in each iterative step of Algorithm \ref{alg:Rec}, we need to solve the following minimization problem
\begin{align}\label{311itePro}
\min_{f_{d}} \Bigg\{ \frac{1}{2}\|S(g^{m} - \mathcal{F}_{mn}(f_{d}))\|_{2}^{2}
+ \frac{1}{2}\|D_{\gamma}^{-1/2}L_{\theta}f_{d}\|_{p^{\theta}}^{p^{\theta}} \Bigg\}.
\end{align}
By introducing a variable $w := D_{\gamma^{k}}^{-1/2}L_{\theta^{k}}f_{d}$, we can rewrite problem (\ref{311itePro}) as follows
\begin{align}\label{311itePro1}
\begin{split}
& \min_{f_{d},w}\Bigg\{
\frac{1}{2}\|S(g^{m} - \mathcal{F}_{mn}(f_{d}))\|_{2}^{2}  + \frac{1}{2}\sum_{\ell \in \Omega_{1}}|w_{\ell}| + \frac{1}{2}\sum_{\ell \in \Omega_{2}}|w_{\ell}|^{2}
\Bigg\} \\
& \text{such that } w = D_{\gamma^{k}}^{-1/2}L_{\theta^{k}}f_{d},
\end{split}
\end{align}
where
\begin{align*}
\Omega_{1} := \Big\{ \ell \, | \, p_{\ell}^{\theta^{k}} = 1, \ell = 0,1,\ldots,n \Big\}, \quad
\Omega_{2} := \Big\{ \ell \, | \, p_{\ell}^{\theta^{k}} = 2, \ell = 0,1,\ldots,n \Big\}.
\end{align*}
Based on this decomposition, we define $w_{1} := w|_{\Omega_{1}}$ and $w_{2} := w|_{\Omega_{2}}$ where
\begin{align*}
\big(w|_{\Omega_{1}}\big)_{\ell} =
\left\{\begin{aligned}
& w_{\ell}, \quad \text{if }\ell \in \Omega_{1}, \\
& 0, \quad\,\,\,\, \text{if }\ell \not\in \Omega_{1},
\end{aligned}\right. \quad\quad
\big(w|_{\Omega_{2}}\big)_{\ell} =
\left\{\begin{aligned}
& w_{\ell}, \quad \text{if }\ell \in \Omega_{2}, \\
& 0, \quad\,\,\,\, \text{if }\ell \not\in \Omega_{2},
\end{aligned}\right.
\end{align*}
with $\ell = 0,1,\ldots,n$.

Now, using splitting technique, we construct an iterative procedure of alternating solving a series of easy subproblems.
The first two problems can be called ``$w$-subproblem'' for fixed $f_{d} = f_{d}^{*}$:
\begin{align}\label{311subPro1}
\mathop{\arg\min}_{w_{1}} \Bigg\{ \frac{1}{2}\sum_{\ell \in \Omega_{1}} |w_{\ell}|
+ \frac{\tilde{\lambda}}{2} \Big\| w_{1} - D_{\gamma^{k}}^{-1/2}L_{\theta^{k}}f_{d}^{*} \Big\|_{2}^{2} \Bigg\},
\end{align}
and
\begin{align}\label{311subPro2}
\mathop{\arg\min}_{w_{2}} \Bigg\{ \frac{1}{2}\sum_{\ell \in \Omega_{2}} |w_{\ell}|^{2}
+ \frac{\tilde{\lambda}}{2} \Big\| w_{2} - D_{\gamma^{k}}^{-1/2}L_{\theta^{k}}f_{d}^{*} \Big\|_{2}^{2} \Bigg\}.
\end{align}
The last subproblem is the ``$f_{d}$-subproblem'' for fixed $w = w^{*}$:
\begin{align}\label{311subPro3}
\mathop{\arg\min}_{f_{d}} \Bigg\{ \frac{1}{2} \Big\|S(g^{m} - \mathcal{F}_{mn}(f_{d})) \Big\|_{2}^{2}
+ \frac{\tilde{\lambda}}{2} \Big\| w^{*} - D_{\gamma^{k}}^{-1/2}L_{\theta^{k}}f_{d} \Big\|_{2}^{2} \Bigg\}.
\end{align}

Using some standard calculations \cite{Chong2011An}, we can easily obtain the minimizer of the subproblems (\ref{311subPro1}) and (\ref{311subPro2}) as follows
\begin{align}\label{311subSol1}
\big(w_{1}\big)_{\ell} =
\left\{\begin{aligned}
& 0, \quad \ell \in \Omega_{1}, \\
& 0, \quad \ell \not\in \Omega_{1} \text{ and }\Big|\big(D_{\gamma^{k}}^{-1/2}L_{\theta^{k}}f_{d}^{*}\big)_{\ell}\Big| \leq \frac{1}{\tilde{\lambda}}, \\
& \Big( \Big| Tf^{*}_{\ell} \Big| - \frac{1}{\tilde{\lambda}} \Big)
\frac{Tf^{*}_{\ell}}{\Big|Tf^{*}_{\ell}\Big|},  \quad \ell \not\in \Omega_{1} \text{ and } \Big|Tf^{*}_{\ell}\Big|>\frac{1}{\tilde{\lambda}},
\end{aligned}\right.
\end{align}
\begin{align}\label{311subSol2}
\big(w_{2}\big)_{\ell} =
\left\{\begin{aligned}
& 0, \quad \ell \in \Omega_{2}, \\
& \frac{\big( D_{\gamma^{k}}^{-1/2}L_{\theta^{k}}f_{d}^{*} \big)_{\ell}}{\tilde{\lambda} + 2}, \quad \ell \in \Omega_{2},
\end{aligned}\right.
\end{align}
with $Tf^{*}_{\ell} := \big(D_{\gamma^{k}}^{-1/2}L_{\theta^{k}}f_{d}^{*}\big)_{\ell}$.

For problem (\ref{311subPro3}), if the the operator $\mathcal{F}_{mn}$ is linear, that is, $\mathcal{F}_{mn} \in \mathbb{R}^{m\times n}$, the
minimization problem (\ref{311subPro3}) is the least squares solution of the linear system
\begin{align}\label{311solLinear}
\left(
  \begin{array}{c}
    S\mathcal{F}_{mn} \\
    \tilde{\lambda}^{1/2}D_{\gamma^{k}}^{-1/2}L_{\theta^{k}} \\
  \end{array}
\right)f_{d} =
\left(
  \begin{array}{c}
    Sg^{m} \\
    \tilde{\lambda}^{1/2}w^{*} \\
  \end{array}
\right),
\end{align}
where $w^{*} = w_{1}^{*} + w_{2}^{*}$. Then, the solution can be obtained by taking pesudo-inverse.
If the operator $\mathcal{F}_{mn}$ is nonlinear, the minimization problem (\ref{311subPro3}) could be seen as a standard Tikhonov regularization problem
which can be solved efficiently, e.g., using iterative solvers \cite{DisBayesian1}.
Based on these considerations, we can show our modified Bregman iterative algorithm in Algorithm \ref{alg:B}.
\begin{algorithm}[htb]
\caption{Modified Bregman iterative algorithm for $T_{f_{d}}(f_{d}, \gamma, \theta)$}
\label{alg:B}
\begin{algorithmic}
\REQUIRE {$\gamma^{k}$, $\theta^{k}$, $p^{\theta^{k}}$, $S$ (\text{Covariance matrix of the noise}), $d$ (Measured data), $N_{\text{max}}$, $\hat{N}_{\text{max}}$, $\tilde{\lambda}$}
\STATE {Set: $f^{0} = f_{d}^{k}$, $g^{0} = 0$; Calculate: $D_{\gamma^{k}}$, $L_{\theta^{k}}$}
\WHILE {$m \leq N_{\text{max}}$}
\STATE {$f^{m,0} \leftarrow f^{m}$}
\FOR {$\tilde{m} = 0$ to $\hat{N}_{\text{max}}$}
\STATE {(1) Calculate $w_{1}^{\tilde{m}+1}$ according to formula (\ref{311subSol1}) with $f_{d}^{*}$ replaced by $f^{m,\tilde{m}}$,}
\STATE {(2) Calculate $w_{2}^{\tilde{m}+1}$ according to formula (\ref{311subSol2}) with $f_{d}^{*}$ replaced by $f^{m,\tilde{m}}$,}
\STATE {(3) Calculate $f^{m,\tilde{m}+1}$ by solving linear system (\ref{311solLinear}) ($\mathcal{F}_{mn}$ is a linear operator)
or using iterative solvers for the classical Tikhonov regularization problem ($\mathcal{F}_{mn}$ is a non-linear operator),}
\ENDFOR
\STATE {$f^{m+1} \leftarrow f^{m,\hat{N}_{\text{max}}}$, $g^{m+1} \leftarrow g^{m} + d - \mathcal{F}_{mn}(f^{m+1})$,}
\ENDWHILE
\STATE {$f_{d}^{k+1} \leftarrow f^{N_{\text{max}}}$,}
\ENSURE {$f_{d}^{k+1}$}
\end{algorithmic}
\end{algorithm}

\subsubsection{Minimization problem for $\theta$}

Remembering formula (\ref{21combineMatrix}), we have
\begin{align}\label{312reduction}
\begin{split}
L_{\theta}f_{d}^{k+1} & = (1-D_{\theta})L_{p}f_{d}^{k+1} + D_{\theta}L_{q}f_{d}^{k+1} \\
& = D_{\theta}(L_{q} - L_{p})f_{d}^{k+1} + L_{p}f_{d}^{k+1}.
\end{split}
\end{align}
Let
\begin{align}\label{312notation1}
Q^{k+1} = \text{diag}\Big( (L_{q}-L_{p})f_{d}^{k+1} \Big), \quad v^{k+1} = L_{p}f_{d}^{k+1},
\end{align}
then equality (\ref{312reduction}) can be written as follows
\begin{align}\label{312redFinalForm}
L_{\theta}f_{d}^{k+1} = Q^{k+1}\theta + v^{k+1}.
\end{align}
Now, we need to solve the following problem
\begin{align}\label{312diffForm}
\min_{\theta \in [0,1]^{n+1}} \, T_{\theta}(f_{d}^{k+1}, \gamma^{k}, \theta),
\end{align}
where
\begin{align}\label{fun1}
T_{\theta}(f_{d}^{k+1}, \gamma^{k}, \theta) =
\frac{1}{2}\Big\| D_{\gamma^{k}}^{-1/2}Q^{k+1}\theta + D_{\gamma^{k}}^{-1/2}v^{k+1} \Big\|_{p^{\theta}}^{p^{\theta}}
+ \frac{1}{2\eta}\Big\| L_{1}\theta \Big\|_{2}^{2}.
\end{align}
The above function (\ref{fun1}) has similar structures as the function $T_{f_{d}}(f_{d}, \gamma^{k}, \theta^{k})$ except the constrain
$\theta \in [0,1]^{n+1}$, so we can use the modified Bregman iterative algorithm proposed in Subsection \ref{subsection311}
to solve the following problem
\begin{align}
\min_{\theta} \, T_{\theta}(f_{d}^{k+1}, \gamma^{k}, \theta).
\end{align}
Then, for $j = 0,1,\ldots,n$, we adjust $\theta_{j} = 0$ if $\theta_{j} < 0$, and $\theta_{j} = 1$ if $\theta_{j} > 1$ as our final solution.
In order to use the modified Bregman iterative algorithm, we employ Algorithm \ref{alg:RecTheta}
which is a recursive procedure similar to Algorithm \ref{alg:Rec}.

\begin{algorithm}[htb]
\caption{Recursive procedure for $\theta^{k+1}$}
\label{alg:RecTheta}
\begin{algorithmic}
\STATE {\textbf{Set: }
\begin{align*}
& \theta_{0} \leftarrow \theta^{k}, \\
& \theta_{1} = \mathop{\arg\min}_{\theta} \Bigg\{ \frac{1}{2\eta} \Big\| L_{1}\theta \Big\|_{2}^{2}
+ \frac{1}{2}\Big\| D_{\gamma^{k}}^{-1/2}Q^{k+1}\theta + D_{\gamma^{k}}^{-1/2}v^{k+1} \Big\|_{p^{\theta_{0}}}^{p^{\theta_{0}}} \Bigg\}, \\
& g^{1} =  - L_{1}\theta_{1},
\end{align*}
}
\STATE {\textbf{Repeat: }
\begin{align*}
& \theta_{m+1} = \mathop{\arg\min}_{\theta} \Bigg\{ \frac{1}{2\eta} \Big\| g^{m} - L_{1}\theta \Big\|_{2}^{2}
+ \frac{1}{2}\Big\| D_{\gamma^{k}}^{-1/2}Q^{k+1}\theta + D_{\gamma^{k}}^{-1/2}v^{k+1} \Big\|_{p^{\theta_{m}}}^{p^{\theta_{m}}} \Bigg\}, \\
& g^{m+1} = g^{m} - L_{1}\theta_{m+1},
\end{align*}
}
\STATE {\textbf{Until: } Some stopping conditions is satisfied.}
\end{algorithmic}
\end{algorithm}
As in the previous subsection, we introduce a variable $w := D_{\gamma^{k}}^{-1/2}Q^{k+1}\theta + D_{\gamma^{k}}^{-1/2}v^{k+1}$.
In each step of Algorithm \ref{alg:RecTheta}, we need to calculate out $p^{\theta_{m}}$ according to formula (\ref{21defineP}), then
we should solve a minimization problem as follows
\begin{align}\label{312eachMinPro1}
\begin{split}
& \min_{\theta, w}\Bigg\{ \frac{1}{2\eta}\|g^{m} - L_{1}\theta\|_{2}^{2} + \frac{1}{2}\sum_{\ell \in \Omega_{1}}|w_{\ell}| + \frac{1}{2}\sum_{\ell\in\Omega_{2}}|w_{\ell}|^{2} \Bigg\} \\
& \text{such that }w = D_{\gamma^{k}}^{-1/2}Q^{k+1}\theta + D_{\gamma^{k}}^{-1/2}v^{k+1},
\end{split}
\end{align}
where
\begin{align*}
\Omega_{1} := \Big\{ \ell \, | \, p_{\ell}^{\theta_{m}} = 1, \ell = 0,1,\ldots,n \Big\}, \quad
\Omega_{2} := \Big\{ \ell \, | \, p_{\ell}^{\theta_{m}} = 2, \ell = 0,1,\ldots,n \Big\}.
\end{align*}
Using same notations as in (\ref{311subPro1}) and (\ref{311subPro2}), we have the following ``$w$-subproblem'' for fixed $\theta = \theta^{*}$:
\begin{align}\label{312subPro1}
\mathop{\arg\min}_{w_{1}} \Bigg\{ \frac{1}{2}\sum_{\ell \in \Omega_{1}} |w_{\ell}|
+ \frac{\tilde{\lambda}}{2} \Big\| w_{1} - D_{\gamma^{k}}^{-1/2}Q^{k+1}\theta^{*} - D_{\gamma^{k}}^{-1/2}v^{k+1} \Big\|_{2}^{2} \Bigg\},
\end{align}
and
\begin{align}\label{312subPro2}
\mathop{\arg\min}_{w_{2}} \Bigg\{ \frac{1}{2}\sum_{\ell \in \Omega_{2}} |w_{\ell}|^{2}
+ \frac{\tilde{\lambda}}{2} \Big\| w_{2} - D_{\gamma^{k}}^{-1/2}Q^{k+1}\theta^{*} - D_{\gamma^{k}}^{-1/2}v^{k+1} \Big\|_{2}^{2} \Bigg\}.
\end{align}
The last subproblem is the ``$\theta$-subproblem'' for fixed $w = w^{*}$:
\begin{align}\label{312subPro3}
\mathop{\arg\min}_{\theta} \Bigg\{ \frac{1}{2\eta} \Big\|g^{m} - L_{1}\theta \Big\|_{2}^{2}
+ \frac{\tilde{\lambda}}{2} \Big\| w^{*} - D_{\gamma^{k}}^{-1/2}Q^{k+1}\theta - D_{\gamma^{k}}^{-1/2}v^{k+1} \Big\|_{2}^{2} \Bigg\}.
\end{align}
Similar to the previous Subsection \ref{subsection311}, we can find the minimizers of the subproblems (\ref{312subPro1}), (\ref{312subPro2}) as follows
\begin{align}\label{312subSol1}
\big(w_{1}\big)_{\ell} =
\left\{\begin{aligned}
& 0, \quad \ell \in \Omega_{1}, \\
& 0, \quad \ell \not\in \Omega_{1} \text{ and }\Big|T\theta^{*}_{\ell}\Big| \leq \frac{1}{\tilde{\lambda}}, \\
& \Big( \Big| T\theta^{*}_{\ell} \Big| - \frac{1}{\tilde{\lambda}} \Big)
\frac{T\theta^{*}_{\ell}}{\Big|T\theta^{*}_{\ell}\Big|},  \quad \ell \not\in \Omega_{1} \text{ and } \Big|T\theta^{*}_{\ell}\Big|>\frac{1}{\tilde{\lambda}},
\end{aligned}\right.
\end{align}
\begin{align}\label{312subSol2}
\big(w_{2}\big)_{\ell} =
\left\{\begin{aligned}
& 0, \quad \ell \in \Omega_{2}, \\
& \frac{ T\theta^{*}_{\ell}}{\tilde{\lambda} + 2}, \quad \ell \in \Omega_{2},
\end{aligned}\right.
\end{align}
with $T\theta^{*}_{\ell} := \big(D_{\gamma^{k}}^{-1/2}Q^{k+1}\theta - D_{\gamma^{k}}^{-1/2}v^{k+1}\big)_{\ell}$.
The minimization problem (\ref{312subPro3}) is the least squares solution of the following linear system
\begin{align}\label{312leastSol}
\left(
  \begin{array}{c}
    \sqrt{\tilde{\lambda}}D_{\gamma^{k}}^{-1/2}Q^{k+1} \\
    \frac{1}{\sqrt{\eta}}L_{1} \\
  \end{array}
\right)\theta
=
\left(
  \begin{array}{c}
    \sqrt{\tilde{\lambda}}\Big( w^{*} - D_{\gamma^{k}}^{-1/2}v^{k+1} \Big) \\
    \frac{1}{\sqrt{\eta}}g^{m} \\
  \end{array}
\right).
\end{align}
Now, for the reader's convenience, we present the modified Bregman iterative algorithm in Algorithm \ref{alg:thetaDetail}.

\begin{algorithm}[htb]
\caption{Modified Bregman iterative algorithm for $T_{\theta}(f_{d}, \gamma, \theta)$}
\label{alg:thetaDetail}
\begin{algorithmic}
\REQUIRE  {$\gamma^{k}$, $f_{d}^{k+1}$, $N_{\text{max}}$, $\hat{N}_{\text{max}}$,
$\tilde{\lambda}$}
\STATE {Set: $\theta_{0} = \theta^{k}$, $g^{0} = 0$; Calculate: $D_{\gamma^{k}}$, $L_{\theta^{k}}$}
\WHILE {$m \leq N_{\text{max}}$}
\STATE {$\theta_{m,0} \leftarrow \theta_{m}$,}
\STATE {Calculate $p^{\theta_{m}}$ according to formula (\ref{21defineP}),}
\FOR {$\tilde{m} = 0$ to $\hat{N}_{\text{max}}$}
\STATE {(1) Calculate $w_{1}^{\tilde{m}+1}$ according to formula (\ref{312subSol1}) with $\theta^{*}$ replaced by $\theta_{m,\tilde{m}}$,}
\STATE {(2) Calculate $w_{2}^{\tilde{m}+1}$ according to formula (\ref{312subSol2}) with $\theta^{*}$ replaced by $\theta_{m,\tilde{m}}$,}
\STATE {(3) Calculate $\theta_{m,\tilde{m}+1}$ by solving linear system (\ref{312leastSol}),}
\ENDFOR
\STATE {$\theta_{m+1} \leftarrow \theta_{m,\hat{N}_{\text{max}}}$, $g^{m+1} \leftarrow g^{m} - L_{1}\,\theta_{m+1}$,}
\ENDWHILE
\STATE {Take $(\theta_{N_{\text{max}}})_{j} = 0$ if $(\theta_{N_{\text{max}}})_{j} < 0$, and take
$(\theta_{N_{\text{max}}})_{j} = 1$ if $(\theta_{N_{\text{max}}})_{j} > 1$ for $j = 0,1,\ldots,n$,}
\STATE {$\theta^{k+1} \leftarrow \theta_{N_{\text{max}}}$,}
\ENSURE {$\theta^{k+1}$}
\end{algorithmic}
\end{algorithm}

\subsubsection{Minimization problem for $\gamma$}

Denote
\begin{align*}
F^{k+1} := L_{\theta^{k+1}}f_{d}^{k+1},
\end{align*}
we find that
\begin{align}\label{313shortNo}
\Big\| D_{\gamma}^{-1/2}L_{\theta^{k+1}}f_{d}^{k+1} \Big\|_{p^{\theta^{k+1}}}^{p^{\theta^{k+1}}} =
\sum_{\ell \in \Omega_{1}} \frac{|F_{\ell}^{k+1}|}{\sqrt{\gamma_{\ell}}}
+ \sum_{\ell \in \Omega_{2}} \frac{(F_{\ell}^{k+1})^{2}}{\gamma_{\ell}},
\end{align}
where
\begin{align*}
\Omega_{1} := \Big\{ \ell \, | \, p_{\ell}^{\theta^{k+1}} = 1, \ell = 0,1,\ldots,n \Big\}, \quad
\Omega_{2} := \Big\{ \ell \, | \, p_{\ell}^{\theta^{k+1}} = 2, \ell = 0,1,\ldots,n \Big\}.
\end{align*}
Combing
\begin{align}\label{313shortTgamma}
\begin{split}
T_{\gamma}(f_{d}^{k+1},\gamma,\theta^{k+1}) := & \frac{1}{2}\|D_{\gamma}^{-1/2}L_{\theta^{k+1}}f_{d}^{k+1}\|_{p^{\theta^{k+1}}}^{p^{\theta^{k+1}}}
+ \sum_{\ell = 0}^{n}\left( \frac{\gamma_{\ell}}{\bar{\gamma}} \right)^{r}  \\
& - \left(r\beta - \frac{3}{2}\right)\sum_{\ell = 0}^{n}\text{log}\gamma_{\ell},
\end{split}
\end{align}
and (\ref{313shortNo}), we know that the formulas of the minimizers are different for $\ell \in \Omega_{1}$ and $\ell \in \Omega_{2}$.
Differentiating $T_{\gamma}(f_{d}^{k+1},\gamma,\theta^{k+1})$ with respect to $\gamma_{\ell}$ when $\ell \in \Omega_{1}$, and setting
the derivative equal to zero, we have
\begin{align}\label{313for1}
-\frac{1}{4}\frac{|F_{\ell}^{k+1}|}{\gamma_{\ell}^{3/2}} + \frac{r \gamma_{\ell}^{r-1}}{\bar{\gamma}^{r}}
- \Big(r\beta - \frac{3}{2}\Big)\frac{1}{\gamma_{\ell}} = 0, \quad \gamma_{\ell} > 0.
\end{align}
For $\ell \in \Omega_{1}$, by similar calculations as above, we find the equality
\begin{align}\label{313for2}
-\frac{1}{2}\frac{(F_{\ell}^{k+1})^{2}}{\gamma_{\ell}^{2}} + \frac{r \gamma_{\ell}^{r-1}}{\bar{\gamma}^{r}}
- \Big(r\beta - \frac{3}{2}\Big)\frac{1}{\gamma_{\ell}} = 0, \quad \gamma_{\ell} > 0.
\end{align}
From (\ref{313for1}) and (\ref{313for2}), we can calculate out $\gamma^{k+1}$ numerically.
For special choices of parameters $r$, $\bar{\gamma}$ and $\beta$, we can obtain explicit formulas, but we omit the details here for concisely.

\subsubsection{Stopping criterion}\label{StoppingCriterionSec}

There are many different choices for the stopping criterion.
Firstly, we can choose the iteration stopping value $\tilde{N}^{1}_{\text{max}}$ such that
\begin{align}\label{32stop1}
\| S(d - \mathcal{F}_{mn}(f_{d}^{\tilde{N}^{1}_{\text{max}}})) \|_{2} \leq \tau
\end{align}
is satisfied first time for some specific $\tau > 1$.
This stopping criterion ensures that we will not incorporate noise contaminated in the data into our inverse results \cite{SplitBregman2009SIAM}.
Secondly, based on the relative change of the norm of the unknowns, we can provide the stopping value as follows
\begin{align}\label{32stop2}
\tilde{N}^{2}_{\text{max}} = \min_{k} \left\{ \sqrt{\Delta_{f}^{k} + \Delta_{\gamma}^{k} + \Delta_{\theta}^{k}} \leq \delta \right\},
\end{align}
where $\delta > 0$ is a given tolerance and
\begin{align*}
\Delta_{f}^{k} = \frac{\|f_{d}^{k} - f_{d}^{k-1}\|_{2}^{2}}{\|f_{d}^{k}\|_{2}^{2}}, \quad
\Delta_{\theta}^{k} = \frac{\|\theta^{k} - \theta^{k-1}\|_{2}^{2}}{\|\theta^{k}\|_{2}^{2}},   \quad
\Delta_{\gamma}^{k} = \frac{\|\gamma^{k} - \gamma^{k-1}\|_{2}^{2}}{\|\gamma^{k}\|_{2}^{2}}.
\end{align*}
Considering both stopping criterion, we can take
\begin{align}\label{32stopFinal}
\tilde{N}_{\text{max}} := \min\left\{ \tilde{N}^{1}_{\text{max}}, \, \tilde{N}^{2}_{\text{max}} \right\}
\end{align}
as our maximum iteration number.

\subsection{Theoretical analysis}\label{theoreticalSection}

In this subsection, we provide some preliminary theoretical analysis for the proposed minimization problem (\ref{3DDARmodel2}),
modified Bregman iterative algorithm and the alternate iterative algorithm shown in Algorithm \ref{alg:A}.

\begin{theorem}\label{existenceUnique}
The minimization problem (\ref{3DDARmodel2}) has a solution and the solution is unique if $T(f_{d},\gamma,\theta)$ is strictly convex.
Each one of the minimization problems (\ref{311miniPro}) and (\ref{312diffForm}) has a unique solution.
\end{theorem}
\begin{proof}
Firstly, we should notice that the positive term $\sum_{j=0}^n(\frac{\gamma_j}{\bar{\gamma}})^r$ with $r > 0$
can control the negative term $-(r\beta-\frac{3}{2})\sum_{j=0}^n\text{log}\gamma_j$ for large enough $\gamma$.
Hence, we obviously find that the function $T(\cdot)$ has a lower bound.
Now, we choose a sequence $\{(f_d^m,\gamma^m,\theta^m)\}_{m=1}^\infty$ such that
\begin{align*}
\lim_{m\rightarrow\infty}T(f_d^m,\gamma^m,\theta^m)=\inf_{f_d,\gamma,\theta}T(f_d,\gamma,\theta).
\end{align*}
For the sequence $\gamma^{m}$, from the boundedness of $T(f_d^m,\gamma^m,\theta^m)$, we have
\begin{align*}
-\infty < \sum_{j = 0}^{n}\left( \frac{\gamma_{j}^{m}}{\bar{\gamma}} \right)^{r}
- \left(r\beta - \frac{3}{2}\right)\sum_{j=0}^{n}\text{log}\gamma_{j}^{m} < \infty,
\end{align*}
which means that there exist two constants $c, C > 0$ such that
\begin{align*}
c \leq \gamma_{j}^{m} \leq C, \quad \text{for } \, j = 1,2,\ldots,n.
\end{align*}
Similarly, we can show that every components of $f_{d}^{m}$ and $\theta^{m}$ are uniformly bounded.
Hence, there exist subsequences $\{f_{d}^{m_{k}}, \gamma^{m_{k}}, \theta^{m_{k}}\}$ such that
\begin{align*}
\lim_{k\rightarrow\infty}f_{d}^{m_{k}} = \bar{f}_{d}, \quad
\lim_{k\rightarrow\infty}\gamma^{m_{k}} = \bar{\gamma}, \quad
\lim_{k\rightarrow\infty}\theta^{m_{k}} = \bar{\theta},
\end{align*}
for some vectors $\bar{f}_{d}$, $\bar{\gamma}$ and $\bar{\theta}$.
Finally, using the continuity of $T(\cdot)$, we obtain
\begin{align}\label{jia21}
\lim_{k\rightarrow\infty}T(f_{d}^{m_{k}},\gamma^{m_{k}},\theta^{m_{k}}) = T(\bar{f}_{d},\bar{\gamma},\bar{\theta})
= \inf_{f_d,\gamma,\theta}T(f_d,\gamma,\theta).
\end{align}
The above formula (\ref{jia21}) indicates that the minimization problem (\ref{3DDARmodel2}) has a solution.
Other claims can be demonstrated similarly, so we omit the proof details.
\end{proof}

\begin{theorem}\label{wellSeq2}
For the modified Bregman iterative algorithm shown in Algorithm \ref{alg:B} with fixed $\gamma^{k}$, $\theta^{k}$ and $S$,
the data fitting error from the iteration is non-increasing, i.e.,
\begin{align}\label{nonDe1}
\|S(d - \mathcal{F}_{mn}(f^{m+1}))\|_{2} \leq \|S(d - \mathcal{F}_{mn}(f^{m}))\|_{2}, \quad m = 1,2,\ldots.
\end{align}
Moreover, it follows that
\begin{align}\label{nonDe2}
\|S(d - \mathcal{F}_{mn}(f^{N_{\text{max}}}))\|_{2}^{2} \leq \frac{2}{N_{max}}F(f_{d}^{*}) + \|S(\mathcal{F}_{mn}(f_{d}^{*})-d)\|_{2}^{2},
\end{align}
where $f_{d}^{*}$ is the true solution of our problem and $F(\cdot)$ defined the same as in (\ref{Fis1}).
\end{theorem}
\begin{proof}
Obviously, we find that
\begin{align}\label{2the1}
\begin{split}
\frac{1}{2}\|S(\mathcal{F}_{mn}(f^{m+1}) - d)\|_{2}^{2} & \leq D_{F}^{p^{m}}(f^{m+1},f^{m}) + \frac{1}{2}\|S(\mathcal{F}_{mn}(f^{m+1}) - d)\|_{2}^{2} \\
& \leq D_{F}^{p^{m}}(f^{m},f^{m}) + \frac{1}{2}\|S(\mathcal{F}_{mn}(f^{m}) - d)\|_{2}^{2} \\
& \leq \frac{1}{2}\|S(\mathcal{F}_{mn}(f^{m}) - d)\|_{2}^{2}.
\end{split}
\end{align}
Through simple calculations, we have
\begin{align*}
D_{F}^{p^{m}}(f_{d},f^{m}) - D_{F}^{p^{m-1}}(f_{d},f^{m-1}) + D_{F}^{p^{m-1}}(f^{m},f^{m-1})
= \langle p^{m} - p^{m-1}, f^{m} - f_{d} \rangle.
\end{align*}
Employing formula (\ref{gradiNext}), we obtain
\begin{align}\label{2the2}
\begin{split}
& \langle f^{m} - f_{d}, p^{m} - p^{m-1} \rangle = \langle f^{m} - f_{d}, -\mathcal{F}_{mn}^{*}S^{T}S(\mathcal{F}_{mn}(f^{m}) - d) \rangle \\
& \quad\quad\quad\quad\quad
= \langle S\mathcal{F}_{mn}(f^{m} - f_{d}), -S(\mathcal{F}_{mn}(f^{m}) - d) \rangle \\
& \quad\quad\quad\quad\quad
= \langle S(\mathcal{F}_{mn}(f_{d})-d), S(\mathcal{F}_{mn}(f^{m})-d) \rangle - \|S(\mathcal{F}_{mn}(f^{m})-d)\|_{2}^{2}  \\
& \quad\quad\quad\quad\quad
\leq \frac{1}{2}\left( \|S(\mathcal{F}_{mn}(f_{d})-d)\|_{2}^{2} - \|S(\mathcal{F}_{mn}(f^{m})-d)\|_{2}^{2} \right)
\end{split}
\end{align}
Take $f_{d} = f_{d}^{*}$ in (\ref{2the2}) and rewrite it as
\begin{align*}
D_{F}^{p^{m}}(f_{d}^{*},f^{m}) & + \frac{1}{2}\|S(\mathcal{F}_{mn}(f^{m})-d)\|_{2}^{2} \\
& \leq D_{F}^{p^{m}}(f_{d}^{*},f^{m}) + D_{F}^{p^{m-1}}(f^{m},f^{m-1}) + \frac{1}{2}\|S(\mathcal{F}_{mn}(f^{m})-d)\|_{2}^{2}  \\
& \leq D_{F}^{p^{m-1}}(f_{d}^{*},f^{m-1}) + \frac{1}{2}\|S(\mathcal{F}_{mn}(f_{d}^{*})-d)\|_{2}^{2}.
\end{align*}
Taking summation for $m = 1,2,\ldots,N_{\text{max}}$ yields
\begin{align}\label{suanFor1}
\begin{split}
D_{F}^{p^{N_{\text{max}}}}(f_{d}^{*},f^{M}) + & \frac{1}{2}\sum_{m = 1}^{N_{\text{max}}} \|S(\mathcal{F}_{mn}(f^{m})-d)\|_{2}^{2} \\
& \leq D_{F}^{p^{0}}(f_{d}^{*},f^{0}) + \frac{1}{2}\sum_{m = 1}^{N_{\text{max}}} \|S(\mathcal{F}_{mn}(f_{d}^{*})-d)\|_{2}^{2}  \\
& \leq F(f_{d}^{*}) + \frac{N_{\text{max}}}{2}\|S(\mathcal{F}_{mn}(f_{d}^{*})-d)\|_{2}^{2}.
\end{split}
\end{align}
Noting that $D_{F}^{p^{N_{\text{max}}}}(f_{d}^{*},f^{N_{\text{max}}}) > 0$, the proof is completed.
\end{proof}


Enlightened by the previous theorem, we can provide the following result for the proposed algorithm
which consists of Algorithm \ref{alg:A}, Algorithm \ref{alg:B}, Algorithm \ref{alg:thetaDetail}, formulas (\ref{313for1}) and (\ref{313for2}).
\begin{theorem}
Algorithm \ref{alg:A} combined with Algorithm \ref{alg:B}, Algorithm \ref{alg:thetaDetail}, formulas (\ref{313for1}) and (\ref{313for2})
generate a sequence $f_{d}^{m}$ for $m = 0,1,\ldots$ satisfying
\begin{align}\label{decreasingAnalysis}
\|S(d - \mathcal{F}_{mn}(f_{d}^{m+1}))\|_{2} \leq \|S(d - \mathcal{F}_{mn}(f_{d}^{m}))\|_{2}.
\end{align}
Moreover, it follows that
\begin{align}\label{decreasingAna2}
\|S(d - \mathcal{F}_{mn}(f_{d}^{\tilde{N}_{\text{max}}}))\|_{2}^{2} \leq \frac{1}{\tilde{N}_{\text{max}}}
\sum_{k = 1}^{\tilde{N}_{\text{max}}}\frac{2}{N_{\text{max}}^{k}}F_{k}(f_{d}^{*}) + \|S(\mathcal{F}_{mn}(f_{d}^{*})-d)\|_{2}^{2},
\end{align}
where $\{N_{\text{max}}^{k}\}_{k = 1,\ldots,\tilde{N}_{\text{max}}}$ represents the maximum iterative number
of each modified Bregman iterative algorithm for solving minimization problem (\ref{311miniPro}) and $F_{k}(f_{d}^{*})$ is
defined as follows
\begin{align}\label{Fis121}
F_{k}(f_{d}^{*}) = \frac{1}{2}\|D_{\gamma^{k}}^{-1/2}L_{\theta^{k}}f_{d}^{*}\|_{p^{\theta^{k}}}^{p^{\theta^{k}}}.
\end{align}
\end{theorem}
\begin{proof}
Considering the results shown in Theorem \ref{wellSeq2}, we easily know that the modified Bregman iterative algorithm
provides a non-increasing iterative sequence. Hence, the inequality (\ref{decreasingAnalysis}) obviously holds.
For every modified Bregman iterative algorithm similar to Algorithm \ref{alg:B}, we can obtain
an inequality similar to (\ref{suanFor1}) as in the proof of Theorem \ref{wellSeq2}.
Adding all the obtained inequalities together, we can easily deduce estimate (\ref{decreasingAna2}).
\end{proof}

\begin{remark}\label{remarkAnalysis1}
Because the inverse problem is considered from Bayes' perspective, we may view $d$ as a random variable.
The randomness of $d$ are caused by the random noise $\epsilon$ which is distributed as $\textbf{Gaussian}(0,\Sigma)$.
Hence, $S(\mathcal{F}_{mn}(f_{d}^{*})-d)$ is a random variable distributed as $\textbf{Gaussian}(0,I)$.
Taking expectations on both sides of (\ref{nonDe2}) and (\ref{decreasingAna2}), we have
\begin{align}\label{nonDe2E}
\mathbb{E}\|S(d - \mathcal{F}_{mn}(f^{N_{\text{max}}}))\|_{2}^{2} \leq \frac{2}{N_{\text{max}}}F(f_{d}^{*}) + 1,
\end{align}
and
\begin{align}\label{jiajia2}
\mathbb{E}\|S(d - \mathcal{F}_{mn}(f_{d}^{\tilde{N}_{\text{max}}}))\|_{2}^{2} \leq \frac{1}{\tilde{N}_{\text{max}}}
\sum_{k = 1}^{\tilde{N}_{\text{max}}}\frac{2}{N_{\text{max}}^{k}}F_{k}(f_{d}^{*}) + 1.
\end{align}
\end{remark}

\section{Applications to some ill-posed inverse problems}

\subsection{Deconvolution problem}

In this subsection, we consider the deconvolution problems with noisy data
\begin{align}\label{41generalModel}
d_{i} = \int_{0}^{1}K(s_{i} - t)f(t)dt + \epsilon_{i}, \quad i = 0, 1, \ldots, m,
\end{align}
where the convolution kernel could be chosen as the Airy function appearing in optical applications
or the Ricker wavelet appearing in seismic explorations.
For the reader's convenience, we list the Airy function and the Ricker wavelet in the following.

\textbf{Airy function:}
\begin{align}\label{41airyFun}
K(t) = A \cdot \left( \frac{J_{1}(\kappa t)}{\kappa t} \right)^{2},
\end{align}
where $J_{1}(\cdot)$ is the Bessel function of first kind of order $1$, and $\kappa$ is a parameter controlling the width of the function,
and $A$ is a parameter controlling the amplitude of the function.

\textbf{Ricker wavelet:}
\begin{align}\label{41rickerWavelet}
K(t) = (1-2\pi^{2}\mathfrak{f}^{2}t^{2})\exp(-\pi^{2}\mathfrak{f}^{2}t^{2}),
\end{align}
where $\mathfrak{f}$ represents peak frequency.

Through introducing the following matrix
\begin{align}\label{41kernelMatirx}
G_{ij} = \frac{1}{\Delta t}K(s_{i} - t_{j}) = \frac{1}{\Delta t}K(\Delta t (i - j)), \quad \Delta t = 1/n,
\end{align}
with $i = 0,1,\ldots,m$ and $j = 0,1,\ldots,n$ and $m \leq n$ and
denoting
\begin{align*}
f_{d} = \{ f_{0}, f_{1}, \ldots, f_{d} \} = \{ f(t_{0}), f(t_{1}), \ldots, f(t_{n}) \}, \quad t_{j} = \frac{j}{n}\Delta t,
\end{align*}
we know that the operator $\mathcal{F}_{mn}$ mentioned in the previous section can be defined as
\begin{align}\label{41operatorF}
\mathcal{F}_{mn}(f_{d}) = G \cdot f_{d},
\end{align}
where $G = \{G_{ij}\}_{1\leq i \leq m, 1\leq j\leq n}$.
Now, we will compare our method with $L^{2}$-norm based Tikhonov regularization method and Total-Variation regularization method.
Concerning the $L^{2}$-norm based Tikhonov regularization method and the Total-Variation regularization method,
we refer to the following two minimization problems
\begin{align}\label{41TikhonovMethod}
\min_{f_{d}} \|d - \mathcal{F}_{mn}(f_{d})\|_{2}^{2} + \lambda_{\text{Tik}} \|L_{1}f_{d}\|_{2}^{2},
\end{align}
\begin{align}\label{41TVMethod}
\min_{f_{d}} \|d - \mathcal{F}_{mn}(f_{d})\|_{2}^{2} + \lambda_{\text{TV}} \|L_{1}f_{d}\|_{1},
\end{align}
where $\lambda_{\text{Tik}}, \lambda_{\text{TV}}$ are two given constants and the matrix $L_{1}$ is defined by formula (\ref{21defineL1}).

\subsubsection{Recover a smooth function}\label{411smoothSec}

Here, we consider a function $f$ defined as follows
\begin{align}\label{411functionDef}
f(t) := \sin(2\pi t).
\end{align}
In the following, we specify $m = n = 500$, $\lambda_{\text{Tik}} = \lambda_{\text{TV}} = 1$.
For the adaptive augmented regularization method, we choose $p = 2$, $q = 0$, $L_{0} = \text{Id}_{n+1}$ and
\begin{align}\label{411parameter}
L_{2} =
\left(
  \begin{array}{ccccc}
    1 & 0 & 0 & \cdots & 0 \\
    -2 & 1 & 0 & \cdots & 0 \\
    1 & -2 & 1 & \cdots & 0 \\
    \vdots & \ddots & \ddots & \ddots & \vdots \\
    0 & 0 & 0 & \cdots & 1
  \end{array}
\right)_{(n+1) \times (n+1).}
\end{align}
For the models of hyper-parameters $\theta$ and $\gamma$, we take
$\bar{\gamma} = 1$, $r = 1$, $\beta = 2$ in formula (\ref{21beliefgamma1}) and $\eta = 1$ in formula (\ref{21beliefTheta1}).
Concerning the modified Bregman iterative algorithm shown in Algorithm \ref{alg:B} and Algorithm \ref{alg:thetaDetail},
we take $N_{\text{max}} = 20$, $\hat{N}_{\text{max}} = 20$ and $\tilde{\lambda} = 10$.
Finally, we choose $\tau = 1.5$ and $\delta = 10^{-3}$ in Section \ref{StoppingCriterionSec} for the stopping criterion
and the noise $\epsilon_{i} \sim \textbf{Gaussian}(0, \sigma\cdot \max(f_{d}))$ with $\sigma = 0.1$.

We choose the Airy function as the convolution kernel and take $\kappa = 1000, A = 500$ in formula (\ref{41airyFun}).
To give the reader an intuitive idea of the forward convolution operator,
we provide the figure of the original function and the measured data with noise in Figure \ref{origData}.
\begin{figure}[htb]
  \centering
  \includegraphics[width = 1\textwidth]{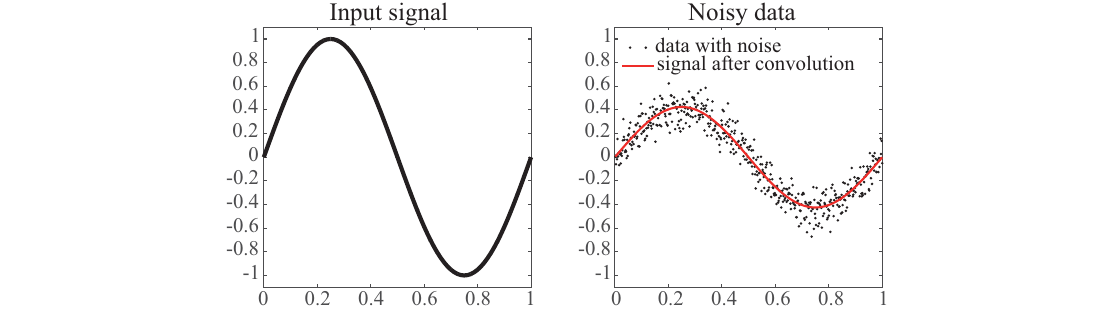}\\
  \caption{The input signal of the computed example (left panel) and the corresponding convolved noisy data (right panel). }\label{origData}
\end{figure}
Under these parameters, we could clearly see the difference between Tikhonov regularization model (\ref{41TikhonovMethod})
and Total-Variation regularization model (\ref{41TVMethod}).
Solving model (\ref{41TikhonovMethod}) and model (\ref{41TVMethod}) by gradient descent algorithm and Bregman iterative algorithm separately,
we can obtain the estimated function which are shown in Figure \ref{TikhonovTVAdaptive}.
\begin{figure}[htb]
  \centering
  \includegraphics[width = 1\textwidth]{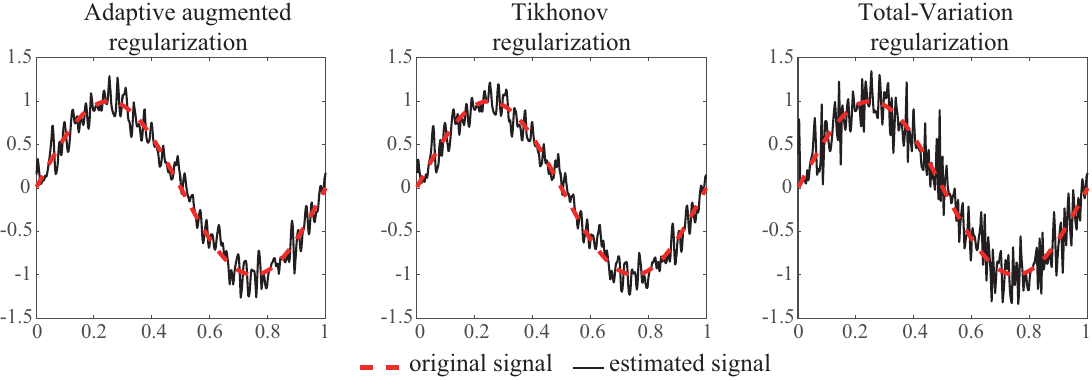}\\
  \caption{The estimated signals obtained by using the AARM (\ref{3DDARmodel1}) (left),
  the Tikhonov regularization model (\ref{41TikhonovMethod}) (middle)
  and the Total-Variation regularization model (\ref{41TVMethod}) (right).
  In each figure, the red curve is the true signal,
  the blue curve is the estimated signal.}\label{TikhonovTVAdaptive}
\end{figure}
From the middle graph and the right graph of Figure \ref{TikhonovTVAdaptive},
we see that the Tikhonov regularization model performs better than the
Total-Variation regularization model under the current settings.
Because the true signal is smooth and without discontinuous points,
the adaptive augmented regularization model proposed in this paper should behaves similar to the Tikhonov regularization model,
which can be clearly seen from the left graph and middle graph in Figure \ref{TikhonovTVAdaptive}.

Here, we should point out that the results shown in Figure \ref{TikhonovTVAdaptive} are actually obtained with not proper regularization
(just take $\lambda_{\text{Tik}} = \lambda_{\text{TV}} = 1$).
For this simple problem, Tikhonov and Total-Variation regularization methods with properly specified regularization parameter $\lambda_{\text{Tik}}$
and $\lambda_{\text{TV}}$ can provide similar recovery results with little difference.
The aim of this example is to show the flexibility of the AARM, so we specify an improper regularization parameter to exaggerate the differences
between the Tikhonov and Total-Variation regularization methods.

In short, this example tells us that the proposed AARM can obtain similar estimates as the Tikhonov regularization model (\ref{41TikhonovMethod})
when the estimated function is smooth.

\subsubsection{Recover a piecewise-constant function}

In this subsection, we consider the following function
\begin{align}\label{412fun}
f(t) :=
\left\{\begin{aligned}
& 0, \quad 0 \leq t < 0.35, \\
& 1, \quad 0.35 \leq t < 0.65, \\
& 0, \quad 0.65 \leq t \leq 1.
\end{aligned}\right.
\end{align}
In order to show the difference visually, we choose the convolution kernel to be the Ricker wavelet with the frequency $\mathfrak{f} = 50$,
take $\sigma = 0.0005$ as the noise level, and all the other parameters are chosen to be the same as in the previous Subsection \ref{411smoothSec}.
Now, we show the original signal and the noisy data in Figure \ref{origData2} and the estimated signals by
AARM, Tikhonov regularization model and Total-Variation regularization model in Figure \ref{2TikhonovTVAdaptive}.
\begin{figure}[htb]
  \centering
  \includegraphics[width = 1\textwidth]{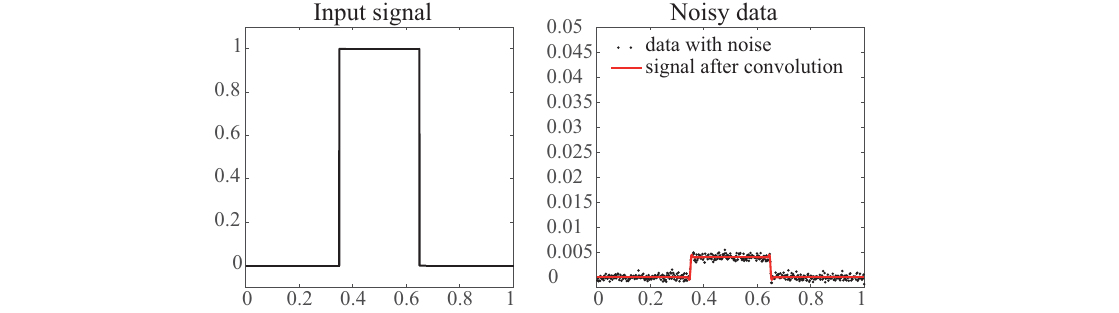}\\
  \caption{The input signal of the computed example (left panel) and the corresponding convolved noisy data (right panel). }\label{origData2}
\end{figure}
\begin{figure}[htb]
  \centering
  \includegraphics[width = 1\textwidth]{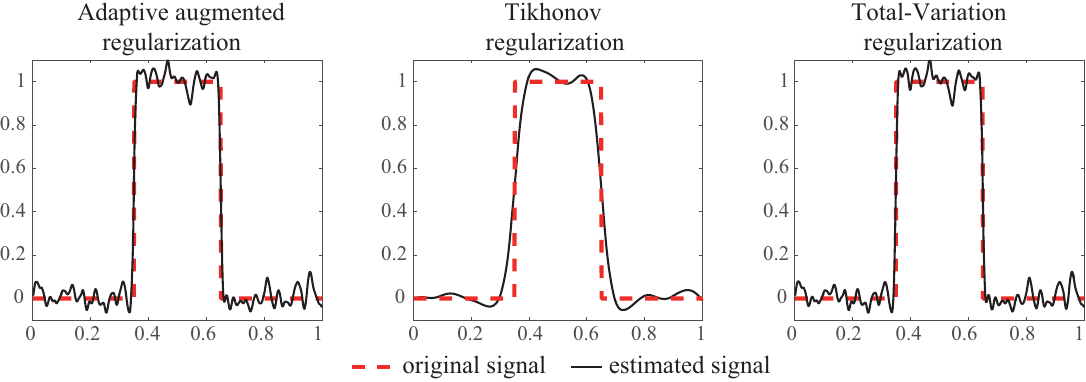}\\
  \caption{The estimated signals obtained by using the AARM (\ref{3DDARmodel1}) (left),
  the Tikhonov regularization model (\ref{41TikhonovMethod}) (middle)
  and the Total-Variation regularization model (\ref{41TVMethod}) (right).
  In each figure, the red curve is the true signal,
  the blue curve is the estimated signal.}\label{2TikhonovTVAdaptive}
\end{figure}
For a piecewise-constant function, the Total-Variation regularization model performs better than the Tikhonov regularization model
which is illustrated in many papers \cite{Osher2005An,RUDIN1992259}.
Since our model incorporate the spatially adaptive mechanism enlightened by the Bayesian inverse framework,
we expect that the AARM should behave like Total-Variation regularization model which can capture the discontinuous changes of a function.
Actually, from the left graph and the right graph in Figure \ref{2TikhonovTVAdaptive},
we find that the AARM generates a similar estimated function as the Total-Variation regularization model
which illustrate the effectiveness of the algorithm proposed in Section \ref{MAPSection}.

\subsubsection{Recover a function with smooth parts and piecewise-constant parts}

In the previous two subsections, for the whole function, we illustrate that the AARM can adjust between Tikhonov regularization model
and Total-Variation regularization model according to the measured data.
However, from the construct procedures of the AARM, we know that this model can characterize local properties of a function
which can adjust its parameters $p^{\theta}$, $\theta$, $\gamma$ automatically at each discrete point.
This adjust procedure makes the AARM performs like Tikhonov regularization model at the smooth parts of the estimated function and
performs as Total-Variation regularization model at the piecewise-constant parts of the estimated function.
Hence, we consider a function defined as follows
\begin{align}\label{413fun}
f(t) :=
\left\{\begin{aligned}
& 0, \quad 0 \leq t < 0.1, \\
& 1, \quad 0.1 \leq t < 0.2,     \\
& 0, \quad 0.2 \leq t < 0.3, \\
& 0.5\cdot \sin(10\pi (t-0.3)), \quad 0.3 \leq t < 0.7,   \\
& 0.3\cdot \sin(100\pi (t-0.5)), \quad 0.7 \leq t \leq 1.
\end{aligned}\right.
\end{align}
This function consists of three parts: for $0 \leq t \leq 0.3$, it is a piecewise-constant function;
for $0.3 \leq t \leq 0.7$, it is a sin function with low frequency;
for $0.7 \leq t \leq 1$, it is a sin function with high frequency.
We choose Airy function with $\kappa = 1000, A = 500$ as the convolution kernel and $\sigma = 0.02$ as the noise level.
All the other parameters are chosen as in Subsection \ref{411smoothSec}.
In order to give the reader an intuitive idea of the function defined in (\ref{413fun}), we show the function and the noisy data in Figure \ref{origData3}.
\begin{figure}[htb]
  \centering
  \includegraphics[width = 1\textwidth]{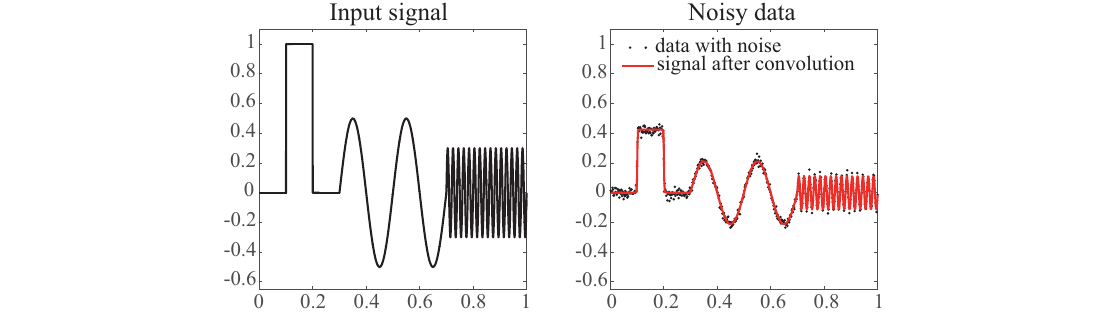}\\
  \caption{The input signal of the computed example (left panel) and the corresponding convolved noisy data (right panel). }\label{origData3}
\end{figure}

Since the function (\ref{413fun}) contains three different parts, we can not visually see the difference of the estimated function
for the Tikhonov regularization model,
the Total-Variation regularization model and the AARM in one figure.
Hence, we only show the recovered function by using the AARM and the original function in the left panel of Figure \ref{RecoveredTheta}
to illustrate that the characteristics of each part of the original function can be captured by our method.
\begin{figure}[htb]
  \centering
  \includegraphics[width = 1\textwidth]{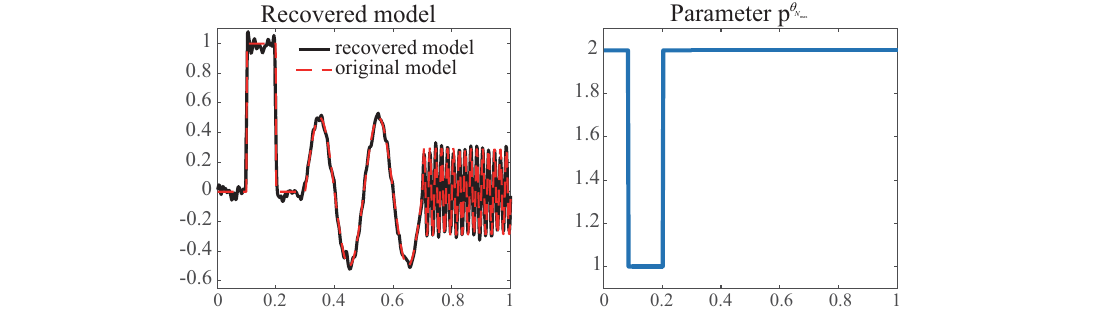}\\
  \caption{The estimated signal of the computed example (left panel) and the parameter $p^{\theta_{\tilde{N}_{\text{max}}}}$ (right panel). }\label{RecoveredTheta}
\end{figure}
In the right panel of Figure \ref{RecoveredTheta}, we provide the value of the parameter $p^{\theta_{\tilde{N}_{\text{max}}}}$
which indicate that the value of $p^{\theta}$ changed according to the measured data efficiently.

At last, we compare the estimated functions provided by the AARM, Tikhonov regularization model and the TV regularization model
in Figure \ref{CompareDDAPComplex}.
\begin{figure}[htb]
  \centering
  \includegraphics[width = 1\textwidth]{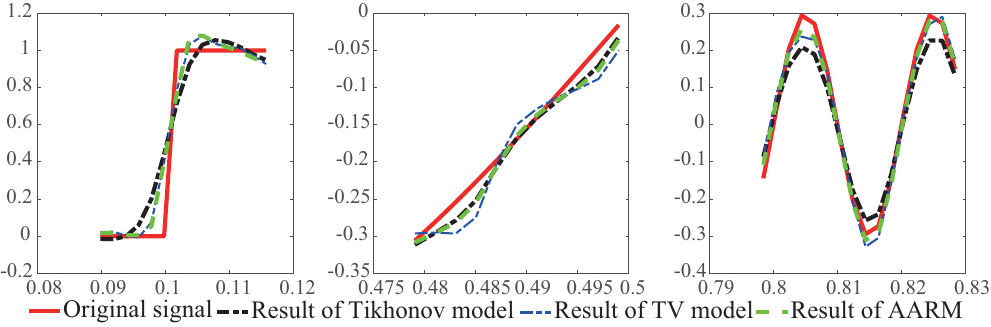}\\
  \caption{Different parts of the original signal and the estimated signal obtained by different methods.
  Red curve represents the original signal.
  Black curve, blue curve and green curve represent the estimated signals obtained by
  Tikhonov regularization model, TV regularization model and AARM separately.
  The left panel shows the result of the discontinuous part, the middle panel shows the result of the continuous part with low frequency,
  and the right panel shows the result of the continuous part with high frequency.
  }\label{CompareDDAPComplex}
\end{figure}
The left panel of Figure \ref{CompareDDAPComplex} shows that the AARM can obtain a similar result as the TV regularization model
when the function is discontinuous in this local region. From this figure, we clearly see that
the estimated function in this local region obtained by AARM or TV model is more likely
to capture sharp changes than the Tikhonov regularization model.
In the middle panel, we show the results in the region of smooth part with low frequency.
Obviously, the AARM provides a function with small oscillation compared with the TV regularization model.
This means that the AARM behaves like Tikhonov regularization model in this local region.
The right panel tells us that the function obtained by the AARM is a little bit better than the function obtained by
the TV regularization model when the function changes rapidly.
Clearly, in general, the AARM and the TV regularization model can provide a better result than the Tikhonov regularization model (\ref{41TikhonovMethod})
in this local region. Actually, in this local region, a Tikhonov regularization model with less smooth constraint than (\ref{41TikhonovMethod}) may provide a better
recovery result since function is highly oscillating in this region (not piecewise constant or slowly varying).
According to the right panel of Figure \ref{RecoveredTheta}, AARM just behaves like a Tikhonov regularization model
with less smooth constraint compared with model (\ref{41TikhonovMethod}).

In summary, the AARM proposed in this paper can adjust its parameters according to the measured data.
This characteristic ensures that the AARM can always provide a good result in each region of
a function with different properties.
The AARM has the capability to characterize local properties of a function rather than provide an average description.

\section{Conclusion}

In this paper, through several hyper-parameters, we construct a prior probability distribution which has the capability
to generate functions with complex behavior. Based on the new prior probability distribution, Bayes' formula are given and
the MAP estimate are also provided. Bearing the connections of Bayesian inverse method and regularization method in mind,
we propose a new regularization model named as the adaptive augmented regularization model which has the ability to
alter its form between various regularization models at each discrete point according to the noisy data.
At last, we construct an alternate iterative algorithm by proposing a modified Bregman iterative algorithm.
The effectiveness of this algorithm has been illustrated through some numerical examples on deconvolution problems.

This work is only a beginning and there are a lot of further interesting problems deserved to be investigated, e.g.,
generalize the AARM to high-dimensional functions, construct more efficient algorithms.


\section*{Acknowledgments}
This work was partially supported by the National Natural Science Foundation of China under the grants no. 11871392, 11501439, 11771347, 11131006 and 41390450.


\bibliographystyle{plain}
\bibliography{references}

\end{document}